\documentclass{article}
\usepackage{cite}
\usepackage{amsmath,amssymb,amsfonts,amsthm}
\usepackage{algorithmic}
\usepackage{graphicx}
\usepackage{algorithm,algorithmic}
\usepackage{hyperref,pdfsync}
\usepackage{textcomp}
\usepackage{dsfont}

\usepackage{epstopdf}
\usepackage{subfigure}
\graphicspath{{./}{./figures/}}

\allowdisplaybreaks

\newtheorem{nnassumption}{\bf Assumption}

\newtheorem{nntheorem}{\bf Theorem}
\newenvironment{theorem}{\begin{nntheorem}\it}{\end{nntheorem}}
\newtheorem{nncorollary}{\bf Corollary}

\newtheorem{nndefinition}{\bf Definition}

\newtheorem{nnproposition}{\bf Proposition}
\newenvironment{proposition}{\begin{nnproposition}\it}{\end{nnproposition}}
\newtheorem{nnproblem}{\bf Problem}

\newtheorem{nnlemma}{\bf Lemma}
\newenvironment{lemma}{\begin{nnlemma}\it}{\end{nnlemma}}
\newtheorem{nnremark}{\bf Remark}
\newenvironment{remark}{\begin{nnremark} \rm }{\hfill \hspace*{1pt}\hfill $\circ$\end{nnremark}}
\newtheorem{nnexample}{\bf Example}

\renewcommand{\leq}{\leqslant}
\renewcommand{\geq}{\geqslant}

\begin{document}
\title{Controllability and Stabilization of a Wave-Heat Cascade System}
\author{Hugo Lhachemi, Christophe Prieur, and Emmanuel Tr{\'e}lat
\thanks{Hugo Lhachemi is with Universit{\'e} Paris-Saclay, CNRS, CentraleSup{\'e}lec, Laboratoire des signaux et syst\`emes, 91190, Gif-sur-Yvette, France (email: hugo.lhachemi@centralesupelec.fr).}
\thanks{Christophe Prieur is with Universit\'e Grenoble Alpes, CNRS, Gipsa-lab, 38000 Grenoble, France (e-mail: christophe.prieur@gipsa-lab.fr).}
\thanks{Emmanuel Trélat is with Sorbonne Universit\'e, Universit\'e Paris Cit\'e, CNRS, Inria, Laboratoire Jacques-Louis Lions, LJLL, F-75005 Paris, France (e-mail: emmanuel.trelat@sorbonne-universite.fr).}}

\date{}

\maketitle

\begin{abstract}
Considering a wave-reaction-diffusion PDE cascade system with wave Neumann control, we first establish controllability properties in a suitable Hilbert space depending on the coupling cascade term. This is done by deriving an observability inequality for the dual problem by resorting to an Ingham-M{\"u}ntz inequality. Second, we design an explicit output feedback control strategy for the actual stabilization of the PDE cascade. The key property is that the underlying operator is a Riesz-spectral operator. 
\end{abstract}


\section*{Introduction}\label{sec:introduction}
\noindent
Let $L>0$, $c\in\mathbb{R}$, and $\beta\in L^\infty(0,L)$ be arbitrary. We consider the control system
\begin{subequations}\label{eq: cascade equation}
    \begin{align}
        &\partial_t y(t,x) = \partial_{xx} y(t,x) + c y(t,x) + \beta(x) z(t,x) , \label{eq: cascade equation - 1} \\
        &\partial_{tt} z(t,x) = \partial_{xx} z(t,x) , \label{eq: cascade equation - 2} \\
        & y(t,0) = y(t,L) = 0 , \label{eq: cascade equation - 3} \\ 
        & z(t,0) = 0 , \ \ \partial_x z(t,L) = u(t) , \label{eq: cascade equation - 4} 
    \end{align}
\end{subequations}
for $t > 0$ and $x \in (0,L)$, where the state $y(t,\cdot):[0,L]\rightarrow\mathbb{R}$ evolves according to the 1D reaction-diffusion equation \eqref{eq: cascade equation - 1}, with a coupling term $\beta(x) z(t,x)$ acting as a source, and the state $z(t,\cdot):[0,L]\rightarrow\mathbb{R}$ is solution of the 1D wave equation \eqref{eq: cascade equation - 2}. The control input $u(t)$ acts on the right Neumann trace of the solution of $z(t,\cdot)$. This kind of wave-reaction-diffusion PDE cascade is encountered, for example, in a simplified model of control of microwave heating \cite{hill1996modelling, wei2012optimal, zhong2014state, celuch2009modeling}. 

Up to our knowledge, the state-of-the-art regarding controllability properties for coupled PDEs of different natures (like here, heat and wave) remains quite limited. 
This is in contrast with parabolic coupled PDEs, for which we refer to \cite{boyer} and references therein.
The model that has much inspired the present article is the control of a heat-wave PDE with coupling at the boundary, studied in \cite{zhang2003polynomial, zhang2004polynomial}. 
Hyperbolic-elliptic couplings have been studied in \cite{rosier2013unique, chowdhury2023boundary}. Stabilization by backstepping control has been developed in \cite{chen2017backstepping, ghousein2020backstepping} for coupled hyperbolic-parabolic PDE systems.

\smallskip
The objective of this paper is twofold. 
In Section~\ref{sec: exact controllability}, we establish exact, exact null and approximate controllability properties for the control system \eqref{eq: cascade equation} in appropriate Hilbert spaces, under sharp assumptions. This is done by establishing an observability inequality for the dual problem by using an instrumental Ingham-M{\"u}ntz inequality. 
Interestingly, the resulting controllability space, which depends on the coupling function $\beta$ and is characterized in a spectral way, is not a conventional functional space. In turn, the control system \eqref{eq: cascade equation} is exponentially stabilizable in that space. 
Since the results of this first section are theoretical (the controllability and stabilization properties are neither constructive nor robust), this motivates Section~\ref{sec: control design}, in which we design a constructive, explicit output feedback control that robustly stabilizes to zero the control system \eqref{eq: cascade equation}. 
The stabilization is proved by a Lyapunov function obtained by spectral considerations, and the feedback control is built with a finite number of modes. 
Extensions to other types of measurements are discussed in Section~\ref{sec: extemsion}.

\section{Controllability and observability properties}\label{sec: exact controllability}
\noindent We define the Hilbert space of complex-valued functions
\begin{subequations}\label{eq: state-space}
    \begin{equation}
        \mathcal{H}^0 = L^2(0,L) \times H_{(0)}^1(0,L) \times L^2(0,L)
    \end{equation}
    where $H_{(0)}^1(0,L) = \{ g \in H^1(0,L)\, \mid\, g(0) = 0 \}$, with the inner product
    \begin{equation}
	\left\langle (f_1,g_1,h_1),(f_2,g_2,h_2)\right\rangle_{\mathcal{H}^0}
        = \int_0^L (f_1 \overline{f_2} + g'_1 \overline{g'_2} + h_1 \overline{h_2} )
    \end{equation}
\end{subequations}
and corresponding norm denoted $\Vert \cdot \Vert_{\mathcal{H}^0}$. When the context is clear, we simply denote by $\langle \cdot , \cdot \rangle$ the inner product of $\mathcal{H}_0$. Setting $\mathcal{X}=(y,z,\partial_t z)^\top$, the control system \eqref{eq: cascade equation} can be written in the form $\dot{\mathcal{X}}(t)=\mathcal{A}_0 \mathcal{X}(t)+\mathcal{B}_0 u(t)$, where $\mathcal{A}_0:D(\mathcal{A}_0)\rightarrow\mathcal{H}^0$ is defined by
\begin{subequations}\label{def_A0}
\begin{equation}
    \mathcal{A}_0 = \begin{pmatrix} \partial_{xx} + c\,\mathrm{id} & \beta\,\mathrm{id} & 0 \\ 0 & 0 & \mathrm{id} \\ 0 & \partial_{xx} & 0 \end{pmatrix}
\end{equation}
with domain
\begin{multline}\label{def_DA0}
    D(\mathcal{A}_0) = \{  (f,g,h)\in H^2(0,L) \times H^2(0,L) \times H^1(0,L)\, \mid\, 
     f(0)=f(L)=g(0)=g'(L)=h(0)=0 \} ,
\end{multline}
\end{subequations}
and the control operator $\mathcal{B}_0$, standing for the right Neumann boundary control \eqref{eq: cascade equation - 4}, can be defined by transposition (see Appendix \ref{sec_app_A} for details). 

It is a nice fact that $\mathcal{A}_0$ is a Riesz operator, as established below, thus allowing us to perform a spectral analysis of the controllability property for \eqref{eq: cascade equation}.

\subsection{Spectral properties of $\mathcal{A}_0$ and of its adjoint}

\begin{lemma}\label{lem: eigenstructures of A0}
    The eigenvalues of $\mathcal{A}_0$ are 
    \begin{align*}
        \lambda_{1,n} = c - \frac{n^2 \pi^2}{L^2}, \; n \in\mathbb{N}^* ,
        \quad \lambda_{2,m} = \frac{(2m+1)i\pi}{2L} , \; m \in\mathbb{Z},
    \end{align*}
    with associated eigenvectors $\phi_{1,n}=(\phi^1_{1,n},\phi^2_{1,n},\phi^3_{1,n})$ and $\phi_{2,m}=(\phi^1_{2,m},\phi^2_{2,m},\phi^3_{2,m})$ respectively given by
    \begin{align}\label{lem: eigenstructures of A0 - parabolic part}
        \phi^1_{1,n}(x) = \sqrt{\frac{2}{L}} \sin\left(\frac{n\pi}{L}x\right), \ \ \phi^2_{1,n}(x) = \phi^3_{1,n}(x) = 0 ,
    \end{align}
    and
    \begin{subequations}\label{lem: eigenstructures of A0 - hyperbolic part}
        \begin{align}\label{lem: eigenstructures of A0 - hyperbolic part - 1}
            & \phi^1_{2,m}(x) = \frac{1}{A_m r_m} \int_x^L \!\! \beta(s) \sinh(\lambda_{2,m} s) \sinh(r_m (x-s)) \,\mathrm{d}s \nonumber \\
            & \! + \! \frac{\sinh(r_m (L-x))}{A_m r_m \sinh(r_m L)} \! \int_0^L \!\!\! \beta(s) \sinh(\lambda_{2,m} s) \sinh(r_m s) \,\mathrm{d}s 
        \end{align}
where $r_m$ is a square root of $\lambda_{2,m} - c$ with nonnegative real part,
        \begin{align}\label{lem: eigenstructures of A0 - hyperbolic part - 2}
            \phi^2_{2,m}(x) = \frac{1}{A_m}\sinh(\lambda_{2,m} x) , \ \ \phi^3_{2,m}(x) = \frac{\lambda_{2,m}}{A_m} \sinh(\lambda_{2,m} x) 
        \end{align}
    \end{subequations}
    with $A_m = \vert \lambda_{2,m} \vert \sqrt{L} =  \frac{\vert 2m+1 \vert \pi}{2\sqrt{L}}$.
\end{lemma}

\smallskip

\begin{proof}
Let $\lambda\in\mathbb{C}$ and $(f,g,h)\in D(\mathcal{A}_0)$ be such that $\mathcal{A}_0(f,g,h)=\lambda (f,g,h)$, i.e., $f,g \in H^2(0,L)$ and $h\in H^1(0,L)$ such that
    \begin{subequations}\label{eq: lemma eigenstructures A0 - system to solve}
        \begin{align}
            & f''+cf+\beta g = \lambda f , \label{eq: lemma eigenstructures A0 - system to solve - 1} \\
            & h = \lambda g , \label{eq: lemma eigenstructures A0 - system to solve - 2} \\
            & g'' = \lambda h = \lambda^2 g , \label{eq: lemma eigenstructures A0 - system to solve - 3} \\
            & f(0) = f(L) = g(0) = g'(L) = h(0) = 0 . \label{eq: lemma eigenstructures A0 - system to solve - 4}
        \end{align}
    \end{subequations}
Assume first that $\lambda = 0$. By \eqref{eq: lemma eigenstructures A0 - system to solve - 3} we have $g''=0$ with $g(0) = g'(L) = 0$, implying along with \eqref{eq: lemma eigenstructures A0 - system to solve - 2} that $g=h=0$. Hence, owing to \eqref{eq: lemma eigenstructures A0 - system to solve - 1}, $f''+cf=0$ and $f(0)=f(L)=0$, from which we deduce that $\frac{n^2 \pi^2}{L^2} = c$ and $f(x) = \sqrt{\frac{2}{L}}\sin\left(\frac{n\pi}{L}x\right)$ for some $n \in\mathbb{N}^*$ (positive integer). In particular, $\lambda = c - \frac{n^2\pi^2}{L^2} = 0$.

Assume now that $\lambda \neq 0$. Using \eqref{eq: lemma eigenstructures A0 - system to solve - 3} and the fact that $g(0)=0$, we have $g(x) = \delta ( e^{\lambda x} - e^{-\lambda x})$ for some $\delta\in\mathbb{R}$. Furthermore, we have $0 = g'(L) = \lambda \delta ( e^{\lambda L} + e^{-\lambda L} )$. Since $\lambda \neq 0$, there are two cases. If $\delta = 0$ then, by \eqref{eq: lemma eigenstructures A0 - system to solve - 2}, $g=h=0$, and by \eqref{eq: lemma eigenstructures A0 - system to solve - 1}, $f''=(\lambda-c)f$ and $f(0)=f(L)=0$ hence $\lambda = c - \frac{n^2\pi^2}{L^2}$ and $f(x) = \sqrt{\frac{2}{L}}\sin\left(\frac{n\pi}{L}x\right)$ for some $n\in\mathbb{N}^*$. If $\delta \neq 0$ then $e^{2\lambda L} = -1$, hence $\lambda = \frac{(2m+1)i\pi}{2L}$ for some $m \in\mathbb{Z}$. Taking without loss of generality $\delta=1/2$, we infer that $g(x) = \sinh(\lambda x)$ and, by \eqref{eq: lemma eigenstructures A0 - system to solve - 2}, $h(x) = \lambda \sinh(\lambda x)$. 

Finally, by \eqref{eq: lemma eigenstructures A0 - system to solve - 1}, we have $f''+(c-\lambda)f=-\beta g$ and $f(0)=f(L)=0$. Noting that $\lambda-c \neq 0$, we denote by $r$ one of its two distinct square roots, i.e., $r^2 = \lambda - c$ with $r \neq 0$. We obtain
$
        f(x) =  \left( \delta_1 - \frac{1}{2r} \int_L^x \beta(s) g(s) e^{-r s} \,\mathrm{d}s \right) e^{r x} 
         + \left( \delta_2 + \frac{1}{2r} \int_L^x \beta(s) g(s) e^{r s} \,\mathrm{d}s \right) e^{-r x}$
for some constants $\delta_1,\delta_2\in\mathbb{C}$ that must be selected so that $f(0)=f(L)=0$. The latter equation yields $\delta_2 = - \delta_1 e^{2rL}$ , implying that
$f(x) =  2 \delta_1 e^{rL} \sinh(r(x-L))  - \frac{1}{r} \int_L^x \beta(s) g(s) \sinh(r(x-s)) \,\mathrm{d}s $.
Then, $f(0)=0$ yields $-2 \delta_1 e^{rL} \sinh(rL) - \frac{1}{r} \int_0^L \beta(s) g(s) \sinh(rs) \,\mathrm{d}s = 0$. We note that $\sinh(rL)=0$ if and only if $e^{2rL}=1$, i.e., if and only if $2rL=2ik\pi$ for some $k\in\mathbb{Z}$. Then, we must have $\lambda = c + r^2 = c - \frac{k^2 \pi^2}{L^2} \in\mathbb{R}$. This is in contradiction with the previously established result imposing $\lambda = \frac{(2m+1)i\pi}{2L} \in i\mathbb{R}\backslash\{0\}$. Hence, $\sinh(rL) \neq 0$ which allows to deduce that $\delta_1 = - \frac{1}{2 e^{rL} r \sinh(r L)} \int_0^L \beta(s) g(s) \sinh(rs) \,\mathrm{d}s$. With the ``scaling'' factor $A_m$, this gives the claimed result.
\end{proof}

\smallskip

\begin{remark}\label{rem_spectrum}
The sequence of eigenvalues $(\lambda_{1,n})_{n\in\mathbb{N}^*}$ (resp., $(\lambda_{1,m})_{m\in\mathbb{Z}}$) is associated with the reaction-diffusion part \eqref{eq: cascade equation - 1} (resp., the wave part \eqref{eq: cascade equation - 2}) of the system \eqref{eq: cascade equation} and is referred to as the \emph{parabolic spectrum} (resp., the \emph{hyperbolic spectrum}). 
%
In what follows, we are going to apply the Ingham-M{\"u}ntz inequality, recalled in Appendix \ref{sec_IM}, so beforehand we need to establish the Riesz basis property.
%
\end{remark}


\smallskip

\begin{lemma}\label{lem: A0 riesz spectral}
    $\Phi = \{\phi_{1,n}\, \mid\, n\in\mathbb{N}^* \} \cup \{\phi_{2,m}\, \mid\, m \in\mathbb{Z} \}$ is a Riesz basis of $\mathcal{H}^0$, of dual Riesz basis $\Psi = \{\psi_{1,n}\, \mid\, n\in\mathbb{N}^* \} \cup \{\psi_{2,m}\, \mid\, m \in\mathbb{Z} \}$. Hence, $\mathcal{A}_0$ is a Riesz spectral operator that generates the $C_0$-semigroup $T_0(t)$ given by
\begin{equation} \label{eq: T0 semigroup}
T_0(t)\mathcal{X} \!=\! \sum_{n\in\mathbb{N}^*} \! e^{\lambda_{1,n}t} \langle \mathcal{X} , \psi_{1,n} \rangle \phi_{1,n}  
         \! +\! \sum_{m \in\mathbb{Z}} \! e^{\lambda_{2,m}t} \langle \mathcal{X} , \psi_{2,m} \rangle \phi_{2,m}
\end{equation}
for every $\mathcal{X}\in\mathcal{H}^0$.
\end{lemma}

\smallskip

\begin{proof}
The proof is done in a more general setting in Lemma~\ref{lem: A riesz spectral} further, so we do not repeat it here.
The formula \eqref{eq: T0 semigroup} follows from \cite[Thm.~2.3.53]{curtain2012introduction}.
\end{proof}

\smallskip

It is well known (see \cite{lions}, \cite[Theorem 5.30]{trelat_SB}, \cite[Theorem 11.2.1]{tucsnakweiss}) that exact controllability for \eqref{eq: cascade equation} is equivalent, by duality, to observability for the dual system $\dot{\mathcal{X}}(t) = \mathcal{A}_0^* \mathcal{X}(t)$ with system output $y(t) = \partial_t z(t,L)$. We thus introduce and study the adjoint operator $\mathcal{A}_0^*$ of $\mathcal{A}_0$.

\smallskip

\begin{lemma}\label{lem: adjoint A0*}
Identifying the Hilbert space $\mathcal{H}^0$ with its dual, 
the adjoint operator $\mathcal{A}_0^*$ is given by
\begin{subequations}\label{def_A0*}
    \begin{equation}
        \mathcal{A}_0^* = \begin{pmatrix} \partial_{xx} + c\,\mathrm{id} & 0 & 0 \\ P_\beta & 0 & - \mathrm{id} \\ 0 & -\partial_{xx} & 0 \end{pmatrix}
    \end{equation}
    with $P_\beta f = \int_0^{(\cdot)} \int_\tau^L \beta(s) f(s) \,\mathrm{d}s\,\mathrm{d}\tau$, of domain
    \begin{multline}
        D(\mathcal{A}_0^*) = \{ (f,g,h)\in H^2(0,L) \times H^2(0,L) \times H^1(0,L)\, \mid\, 
         f(0)=f(L)=g(0)=g'(L)=h(0)=0 \} .
    \end{multline}
\end{subequations}
    Its eigenfunctions are given by the dual Riesz basis $\Psi = \{\psi_{1,n}\, \mid\, n\in\mathbb{N}^* \} \cup \{\psi_{2,m}\, \mid\, m \in\mathbb{Z} \}$ of $\Phi$, where, setting $\psi_{1,n}=(\psi^1_{1,n},\psi^2_{1,n},\psi^3_{1,n})$ and $\psi_{2,m}=(\psi^1_{2,m},\psi^2_{2,m},\psi^3_{2,m})$, 
    \begin{align*}
        & \psi^1_{1,n}(x) = \sqrt{\frac{2}{L}} \sin\left(\frac{n\pi}{L}x\right) , \\ 
        & \psi^2_{1,n}(x) = - \frac{\gamma_n}{\lambda_{1,n}^2 \cosh(\lambda_{1,n} L) } \sqrt{\frac{2}{L}} \cosh(\lambda_{1,n} (x-L)) \\
        & \qquad - \frac{1}{\lambda_{1,n}^2} \sqrt{\frac{2}{L}} \int_x^L \beta(s) \sin\left(\frac{n\pi}{L}s\right) \sinh(\lambda_{1,n} (x-s)) \,\mathrm{d}s   - \frac{1}{\lambda_{1,n}} \sqrt{\frac{2}{L}} \int_0^x \int_L^\tau \beta(s) \sin\left(\frac{n\pi}{L}s\right) \,\mathrm{d}s\,\mathrm{d}\tau , \\ 
        & \psi^3_{1,n}(x) = \frac{\gamma_n}{\lambda_{1,n} \cosh(\lambda_{1,n} L) } \sqrt{\frac{2}{L}} \cosh(\lambda_{1,n} (x-L))  + \frac{1}{\lambda_{1,n}} \sqrt{\frac{2}{L}} \int_x^L \beta(s) \sin\left(\frac{n\pi}{L}s\right) \sinh(\lambda_{1,n} (x-s)) \,\mathrm{d}s ,
    \end{align*}
    with 
    \begin{subequations}\label{def_gamma_n}
    \begin{equation}\label{eq: def gamma_1 - 1}
        \gamma_n = \int_0^L \beta(s) \sin\left(\frac{n\pi}{L}s\right) \sinh(\lambda_{1,n} s) \,\mathrm{d}s
    \end{equation}
    whenever $\frac{n^2 \pi^2}{L^2}\neq c$, i.e., when $\lambda_{1,n} \neq 0$, and 
    \begin{align*}
        \psi^1_{1,n}(x) = & \sqrt{\frac{2}{L}} \sin\left(\frac{n\pi}{L}x\right) , \quad
        \psi^2_{1,n}(x) =  0 , \quad
        \psi^3_{1,n}(x) =  \sqrt{\frac{2}{L}} \int_0^x \int_\tau^L \beta(s) \sin\left(\frac{n\pi}{L}s\right) \,\mathrm{d}s\,\mathrm{d}\tau 
    \end{align*}
    with 
    \begin{equation}\label{eq: def gamma_1 - 2}
        \gamma_n = \int_0^L \int_\tau^L \beta(s) \sin\left(\frac{n\pi}{L}s\right) \,\mathrm{d}s\,\mathrm{d}\tau
    \end{equation}
    \end{subequations}
    whenever $\frac{n^2 \pi^2}{L^2} = c$, i.e., when $\lambda_{1,n} = 0$, and
    \begin{align*}
        & \psi^1_{2,m}(x) = 0 , \quad
        \psi^2_{2,m}(x) = \frac{A_m }{L \vert \lambda_{2,m} \vert^2} \sinh(\lambda_{2,m} x) , \quad
         \psi^3_{2,m}(x) = \frac{A_m \lambda_{2,m}}{L \vert \lambda_{2,m} \vert^2} \sinh(\lambda_{2,m} x) .
    \end{align*}
The eigenvectors satisfy $\mathcal{A}^* \psi_{1,n} = \lambda_{1,n}\psi_{1,n}$ and $\mathcal{A}^* \psi_{2,m} = \bar{\lambda}_{2,m}\psi_{2,m}$ and have been normalized so that $\langle \phi_{1,n} , \psi_{1,n} \rangle = 1$ and $\langle \phi_{2,m} , \psi_{2,m} \rangle = 1$ for all $n\in\mathbb{N}^*$ and $m\in\mathbb{Z}$.
\end{lemma}

\smallskip

\begin{proof}
Using integrations by parts, we have $\langle \mathcal{A}_0(f_1,g_1,h_1) , (f_2,g_2,h_2) \rangle = \langle (f_1,g_1,h_1) , \mathcal{A}_0^* (f_2,g_2,h_2) \rangle$ for all $(f_1,g_1,h_1) \in D(\mathcal{A}_0)$ and $(f_2,g_2,h_2) \in D(\mathcal{A}_0^*)$.
The eigenelements of $\mathcal{A}_0^*$ are computed by solving $\mathcal{A}_0^*(f,g,h) = \lambda (f,g,h)$ for some $\lambda\in\mathbb{C}$ and $(f,g,h)\in D(\mathcal{A}_0^*)$, i.e., 
$$
        f'' + c f  = \lambda f , \qquad
        \int_0^{x} \int_\tau^L \beta(s) f(s) \,\mathrm{d}s\,\mathrm{d}\tau - h  = \lambda g , \qquad
        - g''  = \lambda h .
$$
    Proceeding as in the proof of Lemma~\ref{lem: eigenstructures of A0}, the case $f = 0$ gives $\lambda = \bar{\lambda}_{2,m} = - \lambda_{2,m}$ for some $m\in\mathbb{Z}$ with $(f,g,h) = \psi_{2,m}$ while the case $f \neq 0$ gives $\lambda = \lambda_{1,n}$ for some $n\in\mathbb{N}^*$. We do not give any further details.
\end{proof}

\smallskip

\begin{remark}\label{remgamman}
The eigenelements of $\mathcal{A}_0$ involve the coefficients $\gamma_n$ defined by \eqref{eq: def gamma_1 - 1} if $\frac{n^2 \pi^2}{L^2} \neq c$ and by \eqref{eq: def gamma_1 - 2} if $\frac{n^2 \pi^2}{L^2} = c$. As seen next, these coefficients play a key role in the controllability properties of the control system \eqref{eq: cascade equation}.
It can be noted that $\vert\gamma_n\vert\leq \frac{\mathrm{Cst}}{n^2}e^{n^2\pi^2/L}$.
\end{remark}

\subsection{Controllability}
\noindent
In this section, we establish the exact controllability property of \eqref{eq: cascade equation} in an appropriate functional space $V\subset \mathcal{H}^0$ (of dual $V'$ with respect to the pivot space $\mathcal{H}^0$), the exact null controllability in another space $V_0$ and the approximate controllability in $\mathcal{H}^0$. As alluded above, we proceed by duality, by establishing an exact observability property for the dual system $\dot{\mathcal{X}}(t) = \mathcal{A}_0^* \mathcal{X}(t)$, with observation $\mathcal{B}_0^*\mathcal{X}(t)=\mathcal{X}^3(t,L)$ (see Appendix \ref{sec_app_A}), namely, that for $T>0$ large enough there exists $C_T > 0$ such that 
\begin{equation}\label{eq: observability inequality}
    \int_0^T \vert \mathcal{X}^3(t,L) \vert^2 \,\mathrm{d}t   \geq  C_T \Vert \mathcal{X}(0,\cdot) \Vert_{V'}^2 
\end{equation}
for any solution $t\mapsto \mathcal{X}(t,\cdot)=(\mathcal{X}^1(t,\cdot),\mathcal{X}^2(t,\cdot),\mathcal{X}^3(t,\cdot))^\top$ of the dual system (see Appendix \ref{sec_app_B} for the study of the admissibility property, which is an inequality in the converse way).
Once this inequality will be proved (under appropriate assumptions), we will deduce by duality that the control system \eqref{eq: cascade equation} is exactly controllable in time $T$ in the Hilbert space $V$ (see Theorem \ref{thm: observability inequality} hereafter). The exact null controllability property will be established as well, in another Hilbert space $V_0$, by establishing (under the same assumptions) the finite-time observability inequality
\begin{equation}\label{eq: observability inequality finite time}
    \int_0^T \vert \mathcal{X}^3(t,L) \vert^2 \,\mathrm{d}t   \geq  C_T^0 \Vert \mathcal{X}(T,\cdot) \Vert_{V_0'}^2 ,
\end{equation}
Since its analysis is similar, let us focus on establishing \eqref{eq: observability inequality}.
First, we infer from \eqref{eq: T0 semigroup} that 
$\mathcal{X}(t) = T_0(t)^*\mathcal{X}(0) =  \sum_{n\in\mathbb{N}^*} e^{\lambda_{1,n}t} \langle \mathcal{X}(0) , \phi_{1,n} \rangle \psi_{1,n} 
+ \sum_{m \in\mathbb{Z}} e^{\bar\lambda_{2,m}t} \langle \mathcal{X}(0) , \phi_{2,m} \rangle \psi_{2,m} $,
hence
\begin{equation}\label{eq: observability inequality - system output}
\mathcal{X}^3(t,L) = \sum_{n\in\mathbb{N}^*} e^{\lambda_{1,n}t} \langle \mathcal{X}(0) , \phi_{1,n} \rangle \psi_{1,n}^3(L) 
+ \sum_{m \in\mathbb{Z}} e^{\bar\lambda_{2,m}t} \langle \mathcal{X}(0) , \phi_{2,m} \rangle \psi_{2,m}^3(L) . 
\end{equation}
Therefore, an obvious necessary condition for the observability inequality \eqref{eq: observability inequality} to hold is that $\psi_{1,n}^3(L) \neq 0$ and $\psi_{2,m}^3(L) \neq 0$ for all $n\in\mathbb{N}^*$ and $m \in \mathbb{Z}$. By Lemma~\ref{lem: adjoint A0*} we have $\psi_{1,n}^3(L) = \frac{\gamma_n}{\lambda_{1,n} \cosh(\lambda_{1,n} L) } \sqrt{\frac{2}{L}}$ if $\lambda_{1,n} \neq 0$ and $\psi_{1,n}^3(L) = \gamma_n \sqrt{\frac{2}{L}}$ otherwise, and $\psi_{2,m}^3(L) = \frac{A_m \lambda_{2,m}}{L\vert\lambda_{2,m}\vert^2} \sinh(\lambda_{2,m}L)$. Hence $\psi_{2,m}^3(L)$ is never zero, and $\psi_{1,n}^3(L) \neq 0$ if and only if $\gamma_n \neq 0$. We deduce by duality the following result.

\smallskip

\begin{lemma}\label{eq: necessary condition for observability inequality}
    A necessary condition for the system \eqref{eq: cascade equation} to be exactly controllable in some functional space $V$ is that $\gamma_n \neq 0$ for any $n\in\mathbb{N}^*$, where $\gamma_n$ is defined by \eqref{eq: def gamma_1 - 1} if $\frac{n^2 \pi^2}{L^2}\neq c$ and by \eqref{eq: def gamma_1 - 2} if $\frac{n^2 \pi^2}{L^2} = c$.
\end{lemma}

\smallskip

Actually, the necessary condition stated in Lemma \ref{eq: necessary condition for observability inequality} is sufficient provided $T>2L$. To prove this, we assume the nontriviality of all coefficients $\gamma_n$ and we exploit this assumption to define, in a spectral way, an appropriate Hilbert space $V' \supset \mathcal{H}^0$ in which the observability inequality \eqref{eq: observability inequality} holds true. 

As announced in Remark \ref{rem_spectrum},
by the Ingham-M{\"u}ntz inequality (see Theorem \ref{thm_IM} in Appendix \ref{sec_IM}), 
for any $T > 2L$ there exists $C>0$ such that
\begin{equation}
\int_0^T \vert \mathcal{X}^3(t,L) \vert^2 \,\mathrm{d}t  \geq 
C \sum_{n\in\mathbb{N}^*} \vert \langle \mathcal{X}(0) , \phi_{1,n} \rangle \vert^2 \vert \psi_{1,n}^3(L) \vert^2 e^{2\lambda_{1,n}T} 
+ C \sum_{m\in\mathbb{Z}} \vert \langle \mathcal{X}(0) , \phi_{2,m} \rangle \vert^2 \vert \psi_{2,m}^3(L) \vert^2 . \label{eq: application Ingham-Munsz}
\end{equation}
Using Lemma~\ref{lem: adjoint A0*}, we have
$\psi_{1,n}^3(L) \sim - (2L)^{3/2} e^{cL} \frac{\gamma_n}{n^2\pi^2} e^{-n^2\pi^2/L}$
as $n \rightarrow + \infty$ and $\vert \psi_{2,m}^3(L) \vert \sim \frac{1}{\sqrt{L}}$ as $\vert m \vert \rightarrow + \infty$.
Moreover, $\lambda_{1,n}\sim -\frac{n^2\pi^2}{L^2}$ as $n \rightarrow + \infty$.
This motivates to define 
\begin{equation*}
    V' = \bigg\{ \sum_{n\in\mathbb{N}^*} a_n \psi_{1,n} + \sum_{m \in\mathbb{Z}} b_m \psi_{2,m} \, \mid\,   a_n\in\mathbb{R},\; b_m\in\mathbb{C} ,  \ \ 
     \sum_{n\in\mathbb{N}^*} \vert a_n \vert^2 \dfrac{ \gamma_n^2}{n^4}e^{-\nu n^2}
+ \sum_{m\in\mathbb{Z}} \vert b_m \vert^2 < +\infty   \bigg\} 
\end{equation*}
where
\begin{equation}\label{def_nu}
\nu = 2\frac{\pi^2}{L}\left(1+\frac{T}{L}\right) .
\end{equation}
Endowed with the norm 
$$
\!\!\!  \Vert \mathcal{X} \Vert_{V'}^2 =  
     \sum_{n\in\mathbb{N}^*} \! \vert \langle \mathcal{X} , \phi_{1,n} \rangle \vert^2 \dfrac{ \gamma_n^2}{n^4} e^{-\nu n^2} 
\! + \! \sum_{m\in\mathbb{Z}}\! \vert \langle \mathcal{X} , \phi_{2,m} \rangle \vert^2 ,
$$
$V'$ is a Hilbert space under the assumption that $\gamma_n\neq 0$ for any $n\in\mathbb{N}^*$,
and \eqref{eq: application Ingham-Munsz} yields the observability inequality \eqref{eq: observability inequality}.
Noting that, by Lemma \ref{lem: A0 riesz spectral}, 
\begin{multline*}
\mathcal{H}^0 = \bigg\{  \sum_{n\in\mathbb{N}^*} a_n \phi_{1,n} + \sum_{m \in\mathbb{Z}} b_m \phi_{2,m}  \, \mid\,  a_n\in\mathbb{R},\; b_m\in\mathbb{C} , \ \ 
\sum_{n\in\mathbb{N}^*} \vert a_n \vert^2  + \sum_{m\in\mathbb{Z}} \vert b_m \vert^2 < +\infty   \bigg\} ,
\end{multline*}
by duality, we define the Hilbert space
\begin{subequations}\label{def_V}
\begin{equation}\label{def_V - 1}
V =  \bigg\{  \sum_{n\in\mathbb{N}^*} a_n \phi_{1,n} + \sum_{m \in\mathbb{Z}} b_m \phi_{2,m}  \, \mid\,  a_n\in\mathbb{R},\; b_m\in\mathbb{C} , \ \ 
\sum_{n\in\mathbb{N}^*} \vert a_n \vert^2 \dfrac{n^4}{ \gamma_n^2}e^{\nu n^2} + \sum_{m\in\mathbb{Z}} \vert b_m \vert^2 < +\infty   \bigg\}
\end{equation}
(with $\nu$ defined by \eqref{def_nu})
endowed with the norm 
\begin{equation}
\!\!
\Vert \mathcal{X} \Vert_{V}^2 = \!\! \sum_{n\in\mathbb{N}^*} \! \vert \langle \mathcal{X} , \psi_{1,n} \rangle \vert^2 \dfrac{n^4}{ \gamma_n^2}e^{\nu n^2}  \! + \! \sum_{m\in\mathbb{Z}} \vert \langle \mathcal{X} , \psi_{2,m} \rangle \vert^2 .
\end{equation}
\end{subequations}

To obtain the finite-time observability \eqref{eq: observability inequality finite time} in another appropriate Hilbert space $V_0'$, the analysis is similar. Using the identity before \eqref{eq: observability inequality - system output}, we have
$\langle \mathcal{X}(T) , \phi_{1,n} \rangle = e^{\lambda_{1,n}T} \langle \mathcal{X}(0) , \phi_{1,n} \rangle$
and
$\langle \mathcal{X}(T) , \phi_{2,m}\rangle = e^{\bar\lambda_{2,m}T}\langle \mathcal{X}(0) , \phi_{2,m}\rangle$ 
for all $n\in\mathbb{N}^*$ and $m\in\mathbb{Z}$. Following the same arguments as above, we define the Hilbert spaces $V_0$ and $V_0'$ similarly as $V$ and $V'$ but with a different $\nu$, namely,
$\nu = \frac{2\pi^2}{L}$ (in contrast to \eqref{def_nu}, here $\nu$ does not depend on $T$).
With this definition and the above arguments, we obtain \eqref{eq: observability inequality finite time}.

It follows from the definition of the spaces that
\begin{equation}\label{inclusions1}
\{0\}\times H_{(0)}^1(0,L) \times L^2(0,L)\subset V\subset V_0 
\end{equation}
and, using $\vert\gamma_n\vert\leq \frac{\mathrm{Cst}}{n^2}e^{n^2\pi^2/L}$ (see Remark \ref{remgamman}),  that 
\begin{equation}\label{inclusions2}
V\subset V_0\subset\mathcal{H}^0\subset V_0'\subset V'
\end{equation}
with continuous and dense embeddings, and that $V'$ (resp., $V_0'$) is the dual of $V$ (resp., of $V_0$) with respect to the pivot space $\mathcal{H}^0$.

We obtain the following result.
Recall that $\gamma_n$ is defined by \eqref{eq: def gamma_1 - 1} if $\frac{n^2 \pi^2}{L^2}\neq c$ and by \eqref{eq: def gamma_1 - 2} if $\frac{n^2 \pi^2}{L^2} = c$. 

\smallskip

\begin{theorem}\label{thm: observability inequality}
Assume that $T>2L$ and that
\begin{equation}\label{assumption_gamma_n}
\gamma_n\neq 0\qquad\forall n\in\mathbb{N}^* .
\end{equation}
Then:
\begin{enumerate}
\item\label{item_obs} There exists $C_T>0$ such that the observability inequality \eqref{eq: observability inequality} is satisfied.
\item\label{item_finitetimeobs} There exists $C_T^0>0$ such that the finite-time observability inequality \eqref{eq: observability inequality finite time} is satisfied.
\item\label{item_exact} The control system \eqref{eq: cascade equation} is exactly controllable in time $T$ in the space $V$, i.e., given any $\mathcal{X}_0,\mathcal{X}_1\in V$, there exists $u \in L^2(0,T)$ such that the solution $\mathcal{X}=(y,z,\partial_t z)^\top$ of \eqref{eq: cascade equation} starting at $\mathcal{X}(0)=\mathcal{X}_0$ satisfies $\mathcal{X}(T)=\mathcal{X}_1$.
\item\label{item_exact_null} The control system \eqref{eq: cascade equation} is exactly null controllable in time $T$ in the space $V_0$, i.e., given any $\mathcal{X}_0\in V_0$, there exists $u \in L^2(0,T)$ such that the solution $\mathcal{X}=(y,z,\partial_t z)^\top$ of \eqref{eq: cascade equation} starting at $\mathcal{X}(0)=\mathcal{X}_0$ satisfies $\mathcal{X}(T)=0$.
\item\label{item_approx} The control system \eqref{eq: cascade equation} is approximately controllable in time $T$ in the Hilbert space $\mathcal{H}^0$ (or in any other Hilbert space $H$ such that $H\subset \mathcal{H}^0$ or $\mathcal{H}^0\subset H$ with continuous and dense embeddings), i.e., given any $\mathcal{X}_0,\mathcal{X}_1\in \mathcal{H}^0$ and any $\varepsilon>0$, there exists $u \in L^2(0,T)$ such that the solution $\mathcal{X}=(y,z,\partial_t z)^\top$ of \eqref{eq: cascade equation} starting at $\mathcal{X}(0)=\mathcal{X}_0$ satisfies $\Vert \mathcal{X}(T)-\mathcal{X}_1\Vert_{\mathcal{H}^0}\leq\varepsilon$ .
\end{enumerate}

Moreover, if $T<2L$ or if $\gamma_n=0$ for some $n\in\mathbb{N}^*$ then the control system \eqref{eq: cascade equation} is neither exactly nor exactly null nor approximately controllable in any time $T>0$, in any Hilbert space $H$ such that $H\subset \mathcal{H}^0$ or $\mathcal{H}^0\subset H$ with continuous and dense embeddings.
\end{theorem}

\smallskip

\begin{proof}
The previous discussion, using the Ingham-M\"untz inequality, has shown that the assumptions $T>2L$ and \eqref{assumption_gamma_n} imply \ref{item_obs}) and \ref{item_finitetimeobs}). 
By classical duality arguments due to \cite{lions} (see, e.g., \cite[Theorem 5.30]{trelat_SB}, \cite[Theorem 11.2.1]{tucsnakweiss}), we have \ref{item_obs}) $\Leftrightarrow$ \ref{item_exact}) and \ref{item_finitetimeobs}) $\Leftrightarrow$ \ref{item_exact_null}). 
%
%
The fact that \ref{item_exact}) $\Rightarrow$ \ref{item_approx}) is obtained by density.

If $\gamma_n = 0$ for some $n\in\mathbb{N}^*$ then the 1D subspace $\mathbb{R}\psi_{1,n}$ is not in the closure of the controllability space. This yields half of the last statement of the theorem.
If $T<2L$, the same conclusion is obtained because in this case the wave equation \eqref{eq: cascade equation - 2} strongly loses controllability: actually, the closure of the reachable set has an infinite codimension in $H_{(0)}^1(0,L) \times L^2(0,L)$ (this can be seen, for example, by noting that the set of ``invisible solutions" has infinite dimension, see \cite[Section 2A]{lebeautrelat}).
\end{proof}

\smallskip

\begin{remark}
The conditions $T>2L$ and \eqref{assumption_gamma_n} are almost sharp. Only the limit case where $T=2L$ is not treated. When $T=2L$, the wave equation \eqref{eq: cascade equation - 2} is exactly controllable but the Ingham-M\"untz inequality does not cover this case and we do not know whether the conclusion of Theorem \ref{thm: observability inequality} holds true or not.
\end{remark}

\smallskip

\begin{remark}
Of course, exact controllability in $V$ implies exact null controllability in $V$. But Item \ref{item_exact_null}) of Theorem \ref{thm: observability inequality} gives the stronger exact null controllability in $V_0\supset V$.
\end{remark}

\smallskip

\begin{remark}
In Item \ref{item_exact}) (resp., in Item \ref{item_exact_null})), although the initial condition is taken in the smaller space $V\subset\mathcal{H}^0$ (resp., $V_0\subset\mathcal{H}^0$), the corresponding solution lives in $\mathcal{H}^0$ (by Lemma \ref{lem_adm} in Appendix \ref{sec_app_B}).

Roughly speaking, compared with $\mathcal{H}^0 = L^2(0,L) \times H_{(0)}^1(0,L) \times L^2(0,L)$, the spaces $V$ and $V_0$ are much smaller than $\mathcal{H}^0$ in the $a_n$ components, but, by \eqref{def_V - 1} (see also \eqref{inclusions1}), they coincide with $\mathcal{H}^0$ in the $b_m$ components. This is in accordance with the fact that the observability inequalities \eqref{eq: observability inequality} and \eqref{eq: observability inequality finite time} coincide, when $\mathcal{X}^1\equiv 0$, with the observability inequality for the wave equation $\partial_{tt}\mathcal{X}^3=\partial_{xx}\mathcal{X}^3$, $\mathcal{X}^3(t,0)=\partial_x \mathcal{X}^3(t,L)=0$, namely, $\int_0^T \vert \mathcal{X}^3(t,L)\vert^2\, \mathrm{d}t \geq \mathrm{Cst} \big( \Vert \mathcal{X}^3(0)\Vert_{L^2}^2 + \Vert\partial_t \mathcal{X}^3(0)\Vert_{H^{-1}_{(0)}}^2 \big)$ (where $H^{-1}_{(0)}(0,L)$ is the dual of $H^1_{(0)}(0,L)$ with respect to $L^2(0,L)$) when $T>2L$, 
which holds true by a computation similar to that done in Appendix \ref{sec_app_B}, proof of Lemma \ref{lem_adm}.

So, in some sense, this is only for the reaction-diffusion part that the initial data to be controlled need to be taken in a much smaller space. 
\end{remark}

\smallskip

\begin{remark}
Under \eqref{assumption_gamma_n},
the Hilbert space $V$ defined by \eqref{def_V} (resp., the space $V_0$) is not far from being the largest possible subspace of $\mathcal{H}^0$ in which the control system \eqref{eq: cascade equation} is exactly (resp. null) controllable in time $T>2L$ with controls $u \in L^2(0,T)$. 
This is because the Ingham-M\"untz inequality, and more precisely, the parabolic part, i.e., the ``M\"untz part" of that inequality is not far from being sharp: see Appendix \ref{sec_IM} for a discussion and precise statements.
According to this discussion, the exponential weights in the definition of $V$ and $V_0$ are unavoidable.
\end{remark}

\smallskip

The way we have defined $V$ and $V'$ follows the work \cite{zhang2004polynomial} cited at the beginning. In that paper, the coupling between heat and wave was at the boundary and the resulting controllability space $V$ was defined spectrally with monotone exponentially increasing coefficients (as if $\gamma_n=1$ in \eqref{def_V}). Here, with an internal coupling, our controllability space $V$, given by \eqref{def_V}, is much more complex and unconventional, making this study particularly interesting.
Indeed, depending on the coupling function $\beta\in L^\infty(0,L)$, the coefficients $\gamma_n$ may be wildly oscillating.

It is therefore of interest to consider particular examples of functions $\beta$ and to estimate the resulting coefficients $\gamma_n$, what we do next.

\smallskip

Let us assume that 
\begin{equation}\label{beta_char}
\beta(x) = \beta_0 \, \mathds{1}_{[a,b]}(x) \qquad \forall x\in(0,L)
\end{equation}
for some $\beta_0 \in\mathbb{R}\setminus\{0\}$ and $0 \leq a < b \leq L$. 
Recalling that $\lambda_{1,n} = c - \frac{n^2 \pi^2}{L^2}$, we infer from \eqref{eq: def gamma_1 - 1} that
\begin{subequations}
\label{gamma:def:remark}
\begin{align}
& \gamma_n 
= \frac{\beta_0}{\lambda_{1,n}^2 + \frac{n^2 \pi^2}{L^2}} \bigg( - \frac{n\pi}{L} \sinh( \lambda_{1,n} b ) \cos\left( \frac{n\pi b}{L} \right)  \nonumber \\
& + \frac{n\pi}{L} \sinh( \lambda_{1,n} a ) \cos\Big( \frac{n\pi a}{L} \Big) 
+ \lambda_{1,n} \cosh( \lambda_{1,n} b ) \sin\Big( \frac{n\pi b}{L} \Big)  - \lambda_{1,n} \cosh ( \lambda_{1,n} a ) \sin\Big( \frac{n\pi a}{L} \Big) \bigg)  
\end{align}
if $\frac{n^2 \pi^2}{L^2}\neq c$, and
\begin{equation}
\gamma_n = - \frac{\beta_0 L}{n\pi} \left( b \cos\left(\frac{n\pi}{L} b \right) - a \cos\left(\frac{n\pi}{L} a \right) \right) 
 + \frac{\beta_0 L^2}{n^2\pi^2} \left( \sin\left(\frac{n\pi}{L} b \right) - \sin\left(\frac{n\pi}{L} a \right) \right) 
\end{equation}
\end{subequations}
if $\frac{n^2 \pi^2}{L^2} = c$.
This computation shows the complexity of the coefficients $\gamma_n$ and thus of the functional space $V$, even in this simple case where $\beta$ is the characteristic function of an interval. There are some critical values of $a,b$ for which there exists $n\in\mathbb{N}^*$ such that $\gamma_n=0$ and thus controllability is lost. 
Let us elaborate further.

\smallskip
In the particular case $a=0$ and $b=L$, i.e., $\beta$ constant on $(0,L)$, equal to $\beta_0\neq 0$, we get from \eqref{gamma:def:remark} that
$$
\gamma_n = \frac{(-1)^{n}\beta_0 \frac{n\pi}{L}}{\left( \frac{n^2 \pi^2}{L^2} - c \right)^2 + \frac{n^2 \pi^2}{L^2}} \sinh\left(\Big( \frac{n^2 \pi^2}{L^2} - c \Big) L \right) \neq 0 
$$
if $\frac{n^2 \pi^2}{L^2}\neq c$, and $\gamma_n = \frac{(-1)^{n+1}\beta_0 L^2}{n\pi} \neq 0$ if $\frac{n^2 \pi^2}{L^2} = c$.
Hence, 
the assumption that $\gamma_n\neq 0$ for any $n\in\mathbb{N}^*$ is satisfied and thus the control system \eqref{eq: cascade equation} is exactly null controllable in any time $T>2L$ in the space $V_0$. Moreover, we have $\gamma_n^2\sim \frac{L^6\beta_0^2 e^{-2cL}}{2n^6\pi^6}e^{2n^2\pi^2/L}$ as $n\rightarrow+\infty$, hence the condition 
defining $V_0$ becomes $\sum_{n\in\mathbb{N}^*}n^{10}\vert a_n\vert^2+\sum_{m\in\mathbb{Z}}\vert b_m\vert^2<+\infty$. We thus infer that $V_0$ is a subspace of $H^5(0,L) \times H^1_{(0)}(0,L) \times L^2(0,L)$ (with boundary conditions that can easily be inferred from the differential equations satisfied by the functions $\phi_{1,n}$ and $\phi_{2,m}$).

\smallskip

Let us now take $a=0$ and let us make vary $b$. We also take $L=1$ and $c = 50$. 
Fig~\ref{fig: gamma_2_for_L_1_c_50_a_0} shows $\gamma_2$ as a function of $b \in [0,L]$; we have $\gamma_2=0$ for $b\approx 0.586$ and thus, controllability is lost for this value of $b$. 
\begin{figure}[h]
	\centering
	\includegraphics[width=3.5in]{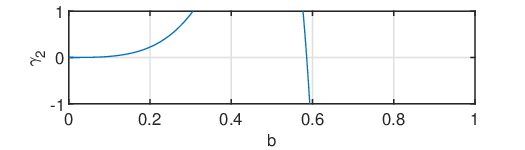}
	\caption{$\gamma_2$ in function of $b \in [0,1]$, with $L=1$, $c=50$ and $a=0$.}
	\label{fig: gamma_2_for_L_1_c_50_a_0}
\end{figure}

\smallskip
In general, we see on \eqref{gamma:def:remark} that $\gamma_n$ may have a wildly oscillating behavior. The coefficients  $n^4 e^{2n^2\pi^2/L}/\gamma_n^2$ defining $V_0$ may thus oscillate as well, depending on the values of $n$, between several powers of $e^{n^2}$. It can be seen that that their asymptotic lower bound, up to scaling, is $e^{\mathrm{Cst} n^2}$ when $b<L$, and $n^{10}$ when $b=L$. Being more precise would require fine arithmetic considerations.

Besides, we have the following genericity result.

\smallskip

\begin{lemma}\label{prop:contr}
Define $S=\{(a,b)\,\mid\, 0\leq a<b\leq L\}$.
The subset $\hat S$ of $S$ such that $\gamma_n\neq 0$ for any $n\in\mathbb{N}^*$ is dense and of full Lebesgue measure in $S$.
\end{lemma}

\smallskip

\begin{proof}
Given any $n\in\mathbb{N}^*$, the set $Z_n=\{(a,b)\in S\,\mid\, \gamma_n=0\}$ is closed, of zero measure and of empty interior because $\gamma_n$ is a nontrivial analytic function of $(a,b)$. 
The union $Z$ of all $Z_n$, $n\in\mathbb{N}^*$, is of zero measure, and of empty interior by the Baire theorem.
The set $\hat S$ is then defined as the complement of $Z$ in $S$.
\end{proof}

\smallskip

\begin{remark}\label{rem_nonrobust}
As a consequence of Theorem \ref{thm: observability inequality} and of Lemma~\ref{prop:contr}, assuming that $T>2L$ and that $\beta$ is defined by \eqref{beta_char} for some $(a,b)\in S$, the control system \eqref{eq: cascade equation} is exactly controllable in time $T$ in the space $V$ if and only if $(a,b)\in \hat S$. In this sense, exact controllability is generic. However, although $\hat S$ is dense and of full Lebesgue measure in $S$, it may fail to be open, and thus robustness of controllability with respect to $(a,b)$ may fail.
\end{remark}

\subsection{An observability property for the system \eqref{eq: cascade equation}}
\noindent
In the previous section, in order to establish an exact controllability property for the control system \eqref{eq: cascade equation}, we have proved the observability property \eqref{eq: observability inequality} for the dual system $\dot{\mathcal{X}}(t) = \mathcal{A}_0^* \mathcal{X}(t)$.
In this section, we establish an additional observability property for the ``primal" system \eqref{eq: cascade equation}.

Due to the cascade nature of \eqref{eq: cascade equation}, any measurement that is solely related to the state $z(t,\cdot)$ of the wave equation cannot provide an observable system (one cannot reconstruct $y(t,\cdot)$ in a univocal way). For this reason, we study the case of a measurement done on the reaction-diffusion part \eqref{eq: cascade equation - 1}, of the form
\begin{equation}\label{eq: measurement heat distributed}
    y_o(t) = \int_0^L c_o(x) y(t,x) \,\mathrm{d}x 
\end{equation}
for some given $c_o \in L^2(0,L)$. Defining
\begin{align*}
	c_{1,n} & = \langle c_o , \phi_{1,n}^1 \rangle_{L^2} = \sqrt{\frac{2}{L}}  \int_0^L c_0(x) \sin\left(\frac{n\pi}{L}x\right) \,\mathrm{d}x , \\ 
	c_{2,m} & = \langle c_o , \phi_{2,m}^1 \rangle_{L^2} = \int_0^L c_0(x) \phi_{2,m}^1(x) \,\mathrm{d}x ,
\end{align*}
for all $n\in\mathbb{N}^*$ and $m\in\mathbb{Z}$,
where $\phi^1_{2,m}$ is defined by \eqref{lem: eigenstructures of A0 - hyperbolic part - 1}, we infer from \eqref{lem: A0 riesz spectral} 
that 
$
y_o(t) =  \sum_{n\in\mathbb{N}^*} c_{1,n} e^{\lambda_{1,n} t} \langle \mathcal{X}(0) , \psi_{1,n} \rangle 
	 + \sum_{m\in\mathbb{Z}} c_{2,m} e^{\lambda_{2,m} t} \langle \mathcal{X}(0) , \psi_{2,m} \rangle $.
This motivates the introduction of the Hilbert space 
\begin{multline*}
    W =  \bigg\{ \sum_{n\in\mathbb{N}^*} a_n \psi_{1,n} + \sum_{m \in\mathbb{Z}} b_n \psi_{2,m} \, \mid\,  a_n\in\mathbb{R},\; b_m\in\mathbb{C} , \ \ 
      \sum_{n\in\mathbb{N}^*} \vert a_n \vert^2 \vert c_{1,n}\vert^2
+ \sum_{m\in\mathbb{Z}} \vert b_m \vert^2 \vert c_{2,m} \vert^2 < +\infty
    \bigg\} 
\end{multline*}
endowed with the norm 
\begin{equation*}
    \Vert (f,g,h) \Vert_{W}^2 =  
     \sum_{n\in\mathbb{N}^*} \vert \langle (f,g,h) , \phi_{1,n} \rangle \vert^2 c_{1,n}^2 
 + \sum_{m\in\mathbb{Z}} \vert \langle (f,g,h) , \phi_{2,m} \rangle \vert^2 \vert c_{2,m} \vert^2 .
\end{equation*}
As before, the Ingham-M{\"u}ntz type inequality leads to the following result.

\smallskip

\begin{proposition}\label{prop:obser}
    A necessary condition for system \eqref{eq: cascade equation} with output \eqref{eq: measurement heat distributed} to be approximately or exactly observable is that $c_{1,n} \neq 0$ for any $n\in\mathbb{N}^*$ and $c_{2,m} \neq 0$ for any $m \in\mathbb{Z}$. Conversely, under the latter conditions, the system is exactly observable in time $T > 2L$ in the space $W$. 
\end{proposition}

\smallskip

\begin{remark}
The condition that all Fourier coefficients $c_{1,n} = \sqrt{\frac{2}{L}}  \int_0^L c_0(x) \sin(\frac{n\pi}{L}x) \,\mathrm{d}x$ are nontrivial is the classical observability condition for a reaction-diffusion equation with Dirichlet boundary conditions. The capability of observing the dynamics of the heat equation through the wave equation via the coupling term $\beta z$ in \eqref{eq: cascade equation} is fully captured by the condition that all coefficients  $c_{2,m} = \int_0^L c_0(x) \phi_{2,m}^1(x) \,\mathrm{d}x$, where $\phi^1_{2,m}$ defined by \eqref{lem: eigenstructures of A0 - hyperbolic part - 1} is linear in the coupling function $\beta$, are nontrivial.
Genericity conditions, similar to those given in Lemma \ref{prop:contr}  and Remark \ref{rem_nonrobust}, can then be established.
\end{remark}

\subsection{Consequence: stabilization}
\noindent 
As a general fact, proved in \cite{liuwangxuyu,trelatwangxu}, exact null controllability implies complete stabilizability (for initial data in the same space). As a consequence of Theorem \ref{thm: observability inequality} (resp., of Proposition \ref{prop:obser}), this means that, under the assumption of nontriviality of the coefficients $\gamma_n$, there exists a linear feedback control $u$ (resp., designed from the output \eqref{eq: measurement heat distributed}) such that the system \eqref{eq: cascade equation}, in closed-loop with this feedback, is exponentially stable for initial data in $V_0$, with a decay rate that can be chosen arbitrarily large. 

Two remarks are in order.

First, as said in Remark \ref{rem_nonrobust}, such a stabilization result is certainly nonrobust, as it can already be seen when $\beta$ is the characteristic function \eqref{beta_char}. This means that, in this case, the feedback control that stabilizes \eqref{eq: cascade equation} for given values of $a$ and $b$ may fail to stabilize the system for slightly perturbed values of $a$ and $b$.

Second, this stabilization result, as a consequence of controllability, is purely theoretical and does not provide an easy way to design an explicit feedback control (unless one is able to invert an infinite-dimensional Gramian, or to solve a Riccati equation). Moreover, stabilization is for initial data in $V_0$ that is, as explained above, a very small space.

This motivates the next section, in which we perform stabilization by a completely different approach, yielding a constructive, explicit and robust feedback control strategy for stabilizing the control system \eqref{eq: cascade equation} in the Hilbert space $\mathcal{H}^0$.

\section{Feedback stabilization}\label{sec: control design}
\noindent The objective of this section is the explicit stabilization of the wave-heat cascade \eqref{eq: cascade equation}, based on two measurements: the heat integral quantity $y_o(t) = \int_0^L c_o(x) y(t,x) \,\mathrm{d}x$, defined by \eqref{eq: measurement heat distributed}, and the wave right-boundary velocity, given by
\begin{equation}\label{eq: measurement wave}
    z_o(t) = \partial_t z(t,L) .
\end{equation} 
Namely, we are going to establish the following result that, at this stage, we state in a rather informal way.
\smallskip

\begin{theorem}
Given any $\delta>0$, there exists an output feedback control, based on the measurements \eqref{eq: measurement heat distributed} and \eqref{eq: measurement wave}, which can be explicitly built from a finite number of modes (depending on $\delta$), such that the closed-loop control system \eqref{eq: cascade equation} is exponentially stable in $\mathcal{H}^0 = L^2(0,L) \times H_{(0)}^1(0,L) \times L^2(0,L)$ and also in $\mathcal{H}^1 = H_0^1(0,L) \times H_{(0)}^1(0,L) \times L^2(0,L)$, with the decay rate $\delta$. 
\end{theorem}

\smallskip

This theorem will be settled in a much more precise way further, in Section \ref{sec_main_feed}, Theorem \ref{thm: main result 2}.

The control strategy consists of first shifting the spectrum of the wave equation \eqref{eq: cascade equation - 2}, by applying a velocity feedback on the right Neumann trace \eqref{eq: cascade equation - 4}, namely, by setting
\begin{equation}\label{eq: preliminary feedback}
    u(t) 
    = -\alpha z_o(t) + v(t)
    = -\alpha \partial_t z(t,L) + v(t)
\end{equation}
where $v$ is another control to be designed, and $\alpha > 1$ is chosen such that (see Remark~\ref{rem: simple eigenvalues})
\begin{equation}\label{eq: condition for selecting alpha}
    c \neq \frac{1}{2L} \log\Big(\frac{\alpha-1}{\alpha+1}\Big) + \frac{n^2\pi^2}{L^2}  \qquad \forall n\in\mathbb{N}^{*} .
\end{equation}
When $v=0$, the output feedback $u$ given by \eqref{eq: preliminary feedback} is known to stabilize the wave equation \eqref{eq: cascade equation - 2}. The additional control $v$ is going to be used and designed from the output \eqref{eq: measurement heat distributed} in order to stabilize the full cascade system \eqref{eq: cascade equation}. Hence, with \eqref{eq: preliminary feedback}, the system is now
\begin{subequations}\label{eq: cascade equation - premilinary feedback}
    \begin{align}
        &\partial_t y(t,x) = \partial_{xx} y(t,x) + c y(t,x) + \beta(x) z(t,x) , \label{eq: cascade equation - premilinary feedback - 1} \\
        &\partial_{tt} z(t,x) = \partial_{xx} z(t,x) , \label{eq: cascade equation - premilinary feedback - 2} \\
        & y(t,0) = y(t,L) = 0 , \label{eq: cascade equation - premilinary feedback - 3} \\ 
        & z(t,0) = 0 , \qquad \partial_x z(t,L) + \alpha \partial_t z(t,L) = v(t) , \label{eq: cascade equation - premilinary feedback - 4} \\
        & y(0,x) = y_0(x) , \; z(0,x) = z_0(x) , \; \partial_t z(0,x) = z_1(x) \label{eq: cascade equation - premilinary feedback - 5}
    \end{align}
\end{subequations}
and the new control is $v$. 

To design an appropriate feedback control $v$, we follow an approach by spectral reduction developed in \cite{coron2004global, coron2006global, lhachemi2020pi, russell1978controllability} in other contexts, which will yield an effective feedback based on a finite number of modes. In turn, we will design a Lyapunov function.

In the Hilbert space $\mathcal{H}^0$ defined by \eqref{eq: state-space}, setting $\mathcal{X}=(y,z,\partial_t z)$ as in Section \ref{sec: exact controllability}, the control system \eqref{eq: cascade equation - premilinary feedback} can be written in the usual form $\dot{\mathcal{X}}(t)=\mathcal{A}\mathcal{X}(t)+\mathcal{B}v(t)$, where $\mathcal{A}:D(\mathcal{A})\rightarrow\mathcal{H}^0$ is defined by
\begin{equation*}
    \mathcal{A} = \begin{pmatrix} \partial_{xx} + c\,\mathrm{id} & \beta\,\mathrm{id} & 0 \\ 0 & 0 & \mathrm{id} \\ 0 & \partial_{xx} & 0 \end{pmatrix}
\end{equation*}
with domain
\begin{multline*}
    D(\mathcal{A}) = \big\{ (f,g,h)\in H^2(0,L) \times H^2(0,L) \times H^1(0,L)\, \mid\,  \\
    f(0)=f(L)=g(0)=h(0)=0 , \   g'(L) + \alpha h(L) = 0 \big\} ,
\end{multline*}
and the control operator $\mathcal{B}$, standing for the right Neumann-like boundary condition \eqref{eq: cascade equation - premilinary feedback - 4}, can be defined by transposition similarly to $\mathcal{B}_0$ in Appendix.

\subsection{Preliminaries: spectral properties of $\mathcal{A}$}\label{sec_prelim_A}

\begin{lemma}\label{lem: eigenstructures of A}
The eigenvalues of $\mathcal{A}$ are 
\begin{align*}
& \lambda_{1,n} = c - \frac{n^2 \pi^2}{L^2}, \quad n\in\mathbb{N}^*, \\ 
& \lambda_{2,m} = \frac{1}{2L} \log\Big(\frac{\alpha-1}{\alpha+1}\Big) + i \frac{m\pi}{L} , \quad m \in\mathbb{Z},
\end{align*}
with associated eigenvectors $\phi_{1,n}=(\phi^1_{1,n},\phi^2_{1,n},\phi^3_{1,n})$ and $\phi_{2,m}=(\phi^1_{2,m},\phi^2_{2,m},\phi^3_{2,m})$ respectively given by
\begin{align*}
\phi^1_{1,n}(x) = \sqrt{\frac{2}{L}} \sin\left(\frac{n\pi}{L}x\right), \quad \phi^2_{1,n}(x) = 0, \quad \phi^3_{1,n}(x) = 0 ,
\end{align*}
and
\begin{multline*}
\phi^1_{2,m}(x) = \frac{1}{A_m r_m} \int_x^L \beta(s) \sinh(\lambda_{2,m} s) \sinh(r_m (x-s)) \,\mathrm{d}s \\
+ \frac{\sinh(r_m (L-x))}{A_m r_m \sinh(r_m L)} \int_0^L \beta(s) \sinh(\lambda_{2,m} s) \sinh(r_m s) \,\mathrm{d}s 
\end{multline*}
where $r_m$ is a square root of $\lambda_{2,m} - c$ with nonnegative real part,
\begin{align*}
\phi^2_{2,m}(x) = \frac{1}{A_m}\sinh(\lambda_{2,m} x) , \quad \phi^3_{2,m}(x) = \frac{\lambda_{2,m}}{A_m} \sinh(\lambda_{2,m} x)  ,
\end{align*}
with $A_m = \frac{1}{L\sqrt{2\mu}} \sqrt{(\mu^2 L^2 + m^2 \pi^2)\sinh(2\mu L)}$
and $\mu = - \frac{1}{2L} \log\big(\frac{\alpha-1}{\alpha+1}\big) > 0$.
\end{lemma}

\begin{remark}\label{rem: simple eigenvalues}
The constraint \eqref{eq: condition for selecting alpha} is introduced to avoid that the real eigenvalue $\lambda_{2,0}$ associated with the wave equation coincides with an eigenvalue $\lambda_{1,n}$ of the reaction-diffusion equation.
\end{remark}

\begin{proof}
Let $\lambda\in\mathbb{C}$ and $(f,g,h)\in D(\mathcal{A})$ be such that $\mathcal{A}(f,g,h)=\lambda (f,g,h)$, i.e., $f,g \in H^2(0,L)$ and $h\in H^1(0,L)$ such that
\begin{subequations}\label{eq: lemma eigenstructures A - system to solve}
    \begin{align}
        & f''+cf+\beta g = \lambda f , \label{eq: lemma eigenstructures A - system to solve - 1} \\
        & h = \lambda g , \label{eq: lemma eigenstructures A - system to solve - 2} \\
        & g'' = \lambda h = \lambda^2 g , \label{eq: lemma eigenstructures A - system to solve - 3} \\
        & f(0) = f(L) = g(0) = h(0) = 0 , \label{eq: lemma eigenstructures A - system to solve - 4} \\
        & g'(L) + \alpha h(L) = g'(L) + \alpha \lambda g(L) = 0 . \label{eq: lemma eigenstructures A - system to solve - 5}
    \end{align}
\end{subequations}
Assume first that $\lambda = 0$. By \eqref{eq: lemma eigenstructures A - system to solve - 3} we have $g''=0$ with $g(0) = g'(L) = 0$, implying along with \eqref{eq: lemma eigenstructures A - system to solve - 2} that $g=h=0$. Hence, from \eqref{eq: lemma eigenstructures A - system to solve - 1}, $f''+cf=0$ and $f(0)=f(L)=0$, from which we deduce that $\frac{n^2 \pi^2}{L^2} = c$ and $f(x) = \sqrt{\frac{2}{L}}\sin\left(\frac{n\pi}{L}x\right)$ for some $n\in\mathbb{N}^*$. In particular, $\lambda = c - \frac{n^2\pi^2}{L^2} = 0$.

Assume now that $\lambda \neq 0$. Using \eqref{eq: lemma eigenstructures A - system to solve - 3} and the fact that $g(0)=0$, we have $g(x) = \delta ( e^{\lambda x} - e^{-\lambda x} )$ for some $\delta\in\mathbb{R}$. Furthermore, by \eqref{eq: lemma eigenstructures A - system to solve - 5},
$0 = g'(L) + \alpha \lambda g(L) = \lambda \delta ( e^{\lambda L} + e^{-\lambda L} + \alpha ( e^{\lambda L} - e^{-\lambda L} ) )$.
Since $\lambda \neq 0$, there are two cases. If $\delta = 0$ then $g=h=0$ and, by \eqref{eq: lemma eigenstructures A - system to solve - 1}, $f''=(\lambda-c)f$ and $f(0)=f(L)=0$ hence $\lambda = c - \frac{n^2\pi^2}{L^2}$ and $f(x) = \sqrt{\frac{2}{L}}\sin\left(\frac{n\pi}{L}x\right)$ for some $n\in\mathbb{N}^*$ provided $\lambda \neq 0$. 
If $\delta \neq 0$ then $e^{\lambda L} + e^{-\lambda L} + \alpha ( e^{\lambda L} - e^{-\lambda L} ) = 0$, which is equivalent to $e^{2\lambda L} = \frac{\alpha-1}{\alpha+1}$. Since $\alpha > 1$, this gives $\lambda = \frac{1}{2L} \log\big(\frac{\alpha-1}{\alpha+1}\big) + i \frac{m\pi}{L}$ for some $m \in\mathbb{Z}$. Moreover $g(x) = \sinh(\lambda x)$ and $h(x) = \lambda \sinh(\lambda x)$, where we take (without loss of generality) $\delta=1/2$. Finally, $f''+(c-\lambda)f=-\beta(x)g$ with $f(0)=f(L)=0$. Recalling that $\alpha > 1$ is selected so that $c \neq \frac{1}{2L} \log\big(\frac{\alpha-1}{\alpha+1}\big)$, we have $\lambda-c \neq 0$. Hence, denoting by $r$ one of its two distinct square roots, i.e., $r^2 = \lambda - c$ with $r \neq 0$, we obtain
$f(x) = \left( \delta_1 - \frac{1}{2r} \int_L^x \beta(s) g(s) e^{-r s} \,\mathrm{d}s \right) e^{r x} + \left( \delta_2 + \frac{1}{2r} \int_L^x \beta(s) g(s) e^{r s} \,\mathrm{d}s \right) e^{-r x}$
for some constants $\delta_1,\delta_2\in\mathbb{C}$ that must be selected such that $f(0)=0$ and $f(L)=0$. The latter equation yields $\delta_2 = - \delta_1 e^{2rL}$ , implying that
$f(x) = 2 \delta_1 e^{rL} \sinh(r(x-L)) - \frac{1}{r} \int_L^x \beta(s) g(s) \sinh(r(x-s)) \,\mathrm{d}s$.
Then, $f(0)=0$ gives $-2 \delta_1 e^{rL} \sinh(rL) - \frac{1}{r} \int_0^L \beta(s) g(s) \sinh(rs) \,\mathrm{d}s = 0$. We note that $\sinh(rL)=0$ if and only if $e^{2rL}=1$, i.e., if and only if $2rL=2ik\pi$ for some $k\in\mathbb{Z}$. Then, we must have $\lambda = c + r^2 = c - \frac{k^2 \pi^2}{L^2}$. Since $\lambda = \frac{1}{2L} \log\big(\frac{\alpha-1}{\alpha+1}\big) + i \frac{m\pi}{L}$, this is possible only if $m=0$, which contradicts the assumption that $\alpha > 1$ has been selected so that $c \neq \frac{1}{2L} \log\big(\frac{\alpha-1}{\alpha+1}\big) + \frac{k^2\pi^2}{L^2}$ for any $k\in\mathbb{N}$. Hence, $\sinh(rL) \neq 0$, which gives $\delta_1 = - \frac{1}{2 e^{rL} r \sinh(r L)} \int_0^L \beta(s) g(s) \sinh(rs) \,\mathrm{d}s$. 
\end{proof}

\smallskip

\begin{remark}
The choice of $\alpha > 1$ satisfying \eqref{eq: condition for selecting alpha} ensures that the eigenvalues of $\mathcal{A}$ are simple.
\end{remark}

\smallskip

\begin{lemma}\label{lem: A riesz spectral}
$\Phi = \{\phi_{1,n}\, \mid\, n\in\mathbb{N}^* \} \cup \{\phi_{2,m}\, \mid\, m \in\mathbb{Z} \}$ is a Riesz basis of $\mathcal{H}^0$. Hence, $\mathcal{A}$ is a Riesz spectral operator that generates a $C_0$-semigroup.
\end{lemma}

\smallskip

\begin{proof}
Noting that $\{\phi^1_{1,n}\, \mid\, n\in\mathbb{N}^* \}$ is a Hilbert basis of $L^2(0,1)$ and that $\{ (\phi^2_{2,m},\phi^3_{2,m})\, \mid\, m\in\mathbb{Z} \}$ is a Riesz basis of $H_{(0)}^1(0,L) \times L^2(0,L)$ we infer, first, that $\Phi$ is $\omega$-linearly independent, and second, defining $\tilde{\phi}_{1,n} = \phi_{1,n}$ and $\tilde{\phi}_{2,m} = (0,\phi^2_{2,m},\phi^3_{2,m})$, that $\tilde{\Phi} = \{\tilde{\phi}_{1,n}\, \mid\, n\in\mathbb{N}^* \} \cup \{\tilde{\phi}_{2,m}\, \mid\, m \in\mathbb{Z} \}$ is a Riesz basis of $\mathcal{H}^0$. By Bari's theorem \cite{gohberg1978introduction}, it then follows that $\Phi$ is a Riesz basis provided that
\begin{equation}\label{bari_ineq}
\sum_{n\in\mathbb{N}^*} \Vert \phi_{1,n} - \tilde{\phi}_{1,n} \Vert_{\mathcal{H}^0}^2 + \sum_{m\in\mathbb{Z}} \Vert \phi_{2,m} - \tilde{\phi}_{2,m} \Vert_{\mathcal{H}^0}^2 
 = \sum_{m\in\mathbb{Z}} \Vert \phi^1_{2,m} \Vert_{L^2}^2
< \infty .
\end{equation}
To prove \eqref{bari_ineq}, we first note that $\phi^1_{2,m} = f_{m,1} + f_{m,2}$ with 
\begin{align*}
& f_{m,1}(x) = \frac{\sinh(r_m (L-x))}{A_m r_m \sinh(r_m L)}  \int_0^x \beta(s) \sinh(\lambda_{2,m} s) \sinh(r_m s) \,\mathrm{d}s , \\
& f_{m,2}(x) = \frac{1}{A_m r_m \sinh(r_m L)} \times \\
& \quad \int_x^L \beta(s) \sinh(\lambda_{2,m} s) \big( \sinh(r_m (L-x) ) \sinh(r_m s)  - \sinh(r_m L) \sinh(r_m (s-x)) \big) \,\mathrm{d}s .
\end{align*}
We study the two terms separately. Since $r_m^2 = \lambda_{2,m} - c = \frac{1}{2L} \log\big(\frac{\alpha-1}{\alpha+1}\big) - c  + i \frac{m\pi}{L}$ with $\operatorname{Re}(r_m)\geq 0$, we infer that $\vert r_m \vert \sim \sqrt{\frac{\vert m \vert \pi}{L}}$ and $\operatorname{Re} r_m \sim \sqrt{\frac{\vert m \vert \pi}{2L}}$ as $\vert m \vert \rightarrow +\infty$. Recall that $\vert \sinh(\operatorname{Re} z) \vert \leq \vert \cosh(z) \vert , \vert \sinh(z) \vert \leq \vert \cosh(\operatorname{Re} z) \vert$ for any $z \in \mathbb{C}$, and $\cosh(x) \leq e^x$ and $\frac{1}{2}(e^x-1) \leq \sinh(x) \leq e^x /2$ for any $x \geq 0$. For $\vert m \vert$ large, we obtain
\begin{align*}
\vert f_{m,1}(x) \vert 
& \leq \frac{\vert \sinh(r_m (L-x)) \vert}{A_m \vert r_m \vert \vert \sinh(r_m L) \vert}  \int_0^x \vert \beta(s) \vert \vert \sinh(\lambda_{2,m} s) \vert \vert \sinh(r_m s) \vert \,\mathrm{d}s \\
& \leq \frac{\Vert \beta \Vert_{L^\infty} \cosh(\operatorname{Re}r_m (L-x))}{A_m \vert r_m \vert \sinh(\operatorname{Re}r_m L)} \int_0^x \cosh(\operatorname{Re}\lambda_{2,m} s) \cosh(\operatorname{Re}r_m s) \,\mathrm{d}s \\
& \leq \frac{2\Vert \beta \Vert_{L^\infty} \cosh(\mu L) e^{\operatorname{Re}r_m (L-x)}}{A_m \vert r_m \vert (e^{\operatorname{Re}r_m L}-1)} \int_0^x \cosh(\operatorname{Re}r_m s) \,\mathrm{d}s \\
& \leq \frac{2\Vert \beta \Vert_{L^\infty} \cosh(\mu L) e^{\operatorname{Re}r_m (L-x)}}{A_m \vert r_m \vert (e^{\operatorname{Re}r_m L}-1)} \times \frac{\sinh(\operatorname{Re}r_m x)}{\operatorname{Re}r_m} \\
& \leq \frac{\Vert \beta \Vert_{L^\infty} \cosh(\mu L) e^{\operatorname{Re}r_m (L-x)}}{A_m \vert r_m \vert (e^{\operatorname{Re}r_m L}-1)} \times \frac{e^{\operatorname{Re}r_m x}}{\operatorname{Re}r_m} \\
& \leq \Vert \beta \Vert_{L^\infty} \cosh(\mu L) \frac{e^{\operatorname{Re}r_m L}}{e^{\operatorname{Re}r_m L}-1} \times \frac{1}{A_m \vert r_m \vert \operatorname{Re}r_m}
\end{align*}
which shows that $\Vert f_{m,1} \Vert_{L^\infty} = O(1/m^2)$ as $\vert m \vert \rightarrow +\infty$. For the second term, noting that $\sinh(r_m (L-x) ) \sinh(r_m s) - \sinh(r_m L) \sinh(r_m (s-x)) = \sinh(r_m x) \sinh(r_m(L-s))$, we have
\begin{align*}
\vert f_{m,2}(x) \vert 
&\leq \frac{1}{A_m \vert r_m \vert \vert \sinh(r_m L) \vert}  \int_x^L \vert \beta(s) \vert \vert \sinh(\lambda_{2,m} s) \vert \vert \sinh(r_m x) \vert \vert \sinh(r_m(L-s)) \vert \,\mathrm{d}s \\
& \leq \frac{\Vert \beta \Vert_{L^\infty}}{A_m \vert r_m \vert \sinh(\operatorname{Re}r_m L)}  \int_x^L \cosh( \operatorname{Re}\lambda_{2,m} s) \cosh(\operatorname{Re}r_m x) \cosh(\operatorname{Re}r_m(L-s)) \,\mathrm{d}s  \\
& \leq \frac{2 \Vert \beta \Vert_{L^\infty} \cosh( \mu L) e^{\operatorname{Re}r_m x}}{A_m \vert r_m \vert ( e^{\operatorname{Re}r_m L}-1)} \int_x^L e^{\operatorname{Re}r_m(L-s)} \,\mathrm{d}s  \\
& \leq \frac{2 \Vert \beta \Vert_{L^\infty} \cosh( \mu L) e^{\operatorname{Re}r_m x}}{A_m \vert r_m \vert ( e^{\operatorname{Re}r_m L}-1)} \times \frac{e^{\operatorname{Re}r_m(L-x)}}{\operatorname{Re}r_m} \\  
& \leq 2 \Vert \beta \Vert_{L^\infty} \cosh( \mu L) \frac{e^{\operatorname{Re}r_m L}}{ e^{\operatorname{Re}r_m L}-1} \times \frac{1}{A_m \vert r_m \vert \operatorname{Re}r_m}
\end{align*}
which shows that $\Vert f_{m,2} \Vert_{L^\infty} = O(1/m^2)$ as $\vert m \vert \rightarrow +\infty$. This completes the proof.
\end{proof}

\smallskip

%

\begin{lemma}\label{lem: dual basis}
The adjoint operator $\mathcal{A}^*$ is given by the same differential matrix operator as $\mathcal{A}_0^*$ in Lemma \ref{lem: adjoint A0*}, but defined on the different domain
\begin{multline*}
D(\mathcal{A^*}) = \{ (f,g,h)\in H^2(0,L) \times H^2(0,L) \times H^1(0,L) \,\mid\,  \\
 f(0)=f(L)=g(0)=h(0)=0 , \ g'(L) - \alpha h(L) = 0 \} .
\end{multline*}
Its eigenfunctions are given by the dual Riesz basis $\Psi = \{\psi_{1,n}\, \mid\, n\in\mathbb{N}^* \} \cup \{\psi_{2,m}\, \mid\, m \in\mathbb{Z} \}$ of $\Phi$, associated with the eigenvalues 
\begin{align*}
& \mu_{1,n} = \lambda_{1,n} = c - \frac{n^2 \pi^2}{L^2}, \qquad n\in\mathbb{N}^* , \\ 
& \mu_{2,m} = \bar\lambda_{2,m} = \frac{1}{2L} \log\big(\frac{\alpha-1}{\alpha+1}\big) - i \frac{m\pi}{L} , \qquad m \in\mathbb{Z} ,
\end{align*}
where, setting $\psi_{1,n}=(\psi^1_{1,n},\psi^2_{1,n},\psi^3_{1,n})$ and $\psi_{2,m}=(\psi^1_{2,m},\psi^2_{2,m},\psi^3_{2,m})$,
\begin{align*}
& \psi^1_{1,n}(x) = \sqrt{\frac{2}{L}} \sin\left(\frac{n\pi}{L}x\right) , \\ 
& \psi^2_{1,n}(x) =  - \frac{\gamma_n}{\lambda_{1,n}^2 \left( \cosh(\lambda_{1,n} L) + \alpha \sinh(\lambda_{1,n} L) \right) } \sqrt{\frac{2}{L}}  \big( \cosh(\lambda_{1,n} (x-L)) - \alpha \sinh(\lambda_{1,n} (x-L)) \big) \\
& \qquad - \frac{1}{\lambda_{1,n}^2} \sqrt{\frac{2}{L}} \int_x^L \beta(s) \sin\left(\frac{n\pi}{L}s\right) \sinh(\lambda_{1,n} (x-s)) \,\mathrm{d}s   - \frac{1}{\lambda_{1,n}} \sqrt{\frac{2}{L}} \int_0^x \int_L^\tau \beta(s) \sin\left(\frac{n\pi}{L}s\right) \,\mathrm{d}s\,\mathrm{d}\tau , \\ 
& \psi^3_{1,n}(x) = \frac{\gamma_n}{\lambda_{1,n} \left( \cosh(\lambda_{1,n} L) + \alpha \sinh(\lambda_{1,n} L) \right) } \sqrt{\frac{2}{L}} \big( \cosh(\lambda_{1,n} (x-L)) - \alpha \sinh(\lambda_{1,n} (x-L)) \big) \\
& \qquad + \frac{1}{\lambda_{1,n}} \sqrt{\frac{2}{L}} \int_x^L \beta(s) \sin\left(\frac{n\pi}{L}s\right) \sinh(\lambda_{1,n} (x-s)) \,\mathrm{d}s ,
\end{align*}
with $\gamma_n = \int_0^L \beta(s) \sin\left(\frac{n\pi}{L}s\right) \sinh(\lambda_{1,n} s) \,\mathrm{d}s$ whenever $\frac{n^2 \pi^2}{L^2}\neq c$, i.e., when $\lambda_{1,n} \neq 0$, and
\begin{align*}
\psi^1_{1,n}(x) = & \sqrt{\frac{2}{L}} \sin\left(\frac{n\pi}{L}x\right) , \quad
\psi^2_{1,n}(x) =  \alpha\gamma_n \sqrt{\frac{2}{L}} x , \quad
\psi^3_{1,n}(x) =  \sqrt{\frac{2}{L}} \int_0^x \int_\tau^L \beta(s) \sin\left(\frac{n\pi}{L}s\right) \,\mathrm{d}s\,\mathrm{d}\tau ,
\end{align*}
with $\gamma_n = \int_0^L \int_\tau^L \beta(s) \sin\left(\frac{n\pi}{L}s\right) \,\mathrm{d}s\,\mathrm{d}\tau$ whenever $\frac{n^2 \pi^2}{L^2} = c$, i.e., when $\lambda_{1,n} = 0$, and
\begin{align*}
& \psi^1_{2,m}(x) = 0 , \quad
\psi^2_{2,m}(x) = \frac{A_m}{L (\bar{\lambda}_{2,m})^2} \sinh(\bar{\lambda}_{2,m} x) , \quad
 \psi^3_{2,m}(x) = - \frac{A_m}{L \bar{\lambda}_{2,m}} \sinh(\bar{\lambda}_{2,m} x) .
\end{align*}
The eigenvectors satisfy $\mathcal{A}^* \psi_{1,n} = \lambda_{1,n}\psi_{1,n}$ and $\mathcal{A}^* \psi_{2,m} = \bar{\lambda}_{2,m}\psi_{2,m}$ and have been normalized so that $\langle \phi_{1,n} , \psi_{1,n} \rangle = 1$ and $\langle \phi_{2,m} , \psi_{2,m} \rangle = 1$ for all $n\in\mathbb{N}^*$ and $m\in\mathbb{Z}$.
\end{lemma}

\smallskip

\begin{proof}
We proceed similarly as in the proofs of Lemmas \ref{lem: adjoint A0*} and \ref{lem: eigenstructures of A}, studying separately the cases $\lambda_{1,n} = 0$ and $\lambda_{1,n} \neq 0$. Note that, in the latter case, the introduced eigenfunctions are well defined because, by \eqref{eq: condition for selecting alpha}, $\cosh(\lambda_{1,n} L) + \alpha \sinh(\lambda_{1,n} L) \neq 0$. We do not provide further details.
\end{proof}

\smallskip

Lemma~\ref{lem: A riesz spectral} shows that, like $\mathcal{A}_0$, $\mathcal{A}$ is a Riesz operator. This allows us to perform stabilization in $L^2 \times H^1 \times L^2$ norm. But actually, the stabilization of the parabolic part of the system can also be achieved in $H^1$ norm thanks to the following result.

\smallskip

\begin{lemma}\label{lemma: Riesz basis H1}
	$\Phi^1 = \{\frac{L}{n\pi}\phi_{1,n}\, \mid\, n\in\mathbb{N}^* \} \cup \{\phi_{2,m}\, \mid\, m \in\mathbb{Z} \}$ is a Riesz basis of the Hilbert space
	\begin{subequations}\label{eq: space H1}
	\begin{equation}
		\mathcal{H}^1 = H_0^1(0,L) \times H_{(0)}^1(0,L) \times L^2(0,L)
	\end{equation}
	endowed with the inner product 
	\begin{equation}
        \langle (f_1,g_1,h_1) , (f_2,g_2,h_2) \rangle = \int_0^L (f_1'\overline{f_2'} + g'_1 \overline{g'_2} + h_1\overline{h_2} ) .
    \end{equation}
\end{subequations}
\end{lemma}

\begin{proof}
Denoting by $\triangle:H^1_0(0,L) \cap H^2(0,L)\rightarrow L^2(0,L)$ the usual Dirichlet operator, the operator
$J  :  \mathcal{H}^1 \rightarrow \mathcal{H}^0$ defined by $J(f,g,h)=(\sqrt{-\triangle}f , g ,h)$ is easily seen to be a surjective isometry. 
Hence $\{J^{-1}\phi_{1,n}\, \mid\, n\in\mathbb{N}^* \} \cup \{J^{-1}\phi_{2,m}\, \mid\, m \in\mathbb{Z} \}$ is a Riesz basis of $\mathcal{H}^1$, and we have $J^{-1}\phi_{1,n} = \frac{L}{n\pi} \phi_{1,n}$ and $J^{-1}\phi_{2,m} = ( (\sqrt{-\triangle})^{-1} \phi_{2,m}^1 , \phi_{2,m}^2 , \phi_{2,m}^3 )$. To use the Bari theorem, we have to show that
\begin{equation*}
    \sum_{m\in\mathbb{Z}} \Vert \phi^1_{2,m} - (\sqrt{-\triangle})^{-1}\phi^1_{2,m} \Vert_{H_0^1}^2
    < +\infty .
\end{equation*}
Since $\Vert (\sqrt{-\triangle})^{-1}\phi^1_{2,m} \Vert_{H_0^1} = \Vert \phi^1_{2,m} \Vert_{L^2}$, it is sufficient to show that $\sum_{m\in\mathbb{Z}} \Vert \phi^1_{2,m} \Vert_{H_0^1}^2 < +\infty$. Following the proof of Lemma~\ref{lem: A riesz spectral}, we have $\phi_{2,m}^1 = f_{m,1} + f_{m_2}$ with $\Vert f_{m,1}' \Vert_{L^\infty} = O(1/\vert m \vert^{3/2})$ and $\Vert f_{m,2}' \Vert_{L^\infty} = O(1/\vert m \vert^{3/2})$. The lemma follows. 
\end{proof}

\subsection{Finite-dimensional reduced model}
\noindent 
In order to work with an homogeneous representation of \eqref{eq: cascade equation - premilinary feedback}, we make the change of variables
\begin{subequations}\label{eq: change of variable}
\begin{align}
& w^1(t,x) = y(t,x) , \quad w^2(t,x) = z(t,x) , \\
& w^3(t,x) = \partial_t z(t,x) - \frac{x}{\alpha L} v(t) .
\end{align}
\end{subequations}
Assuming that $v$ is of class $\mathcal{C}^1$ (this assumption will be fulfilled by the upcoming adopted feedback control strategy), the control system becomes
\begin{subequations}\label{eq: cascade equation - homonegenous}
\begin{align}
& \partial_t w^1 = \partial_{xx} w^1 + c w^1 + \beta w^2 , \\
& \partial_t w^2 = w^3 + \frac{x}{\alpha L} v , \\
& \partial_t w^3 = \partial_{xx} w^2 - \frac{x}{\alpha L} \dot{v} , \\
& w^1(t,0) = w^1(t,L) = 0 , \\ 
& w^2(t,0) = 0 , \qquad \partial_x w^2(t,L) + \alpha w^3(t,L) = 0 ,\\
& w^1(0,x) = y_0(x) , \qquad w^2(0,x) = z_0(x) , \\ 
& w^3(0,x) = z_1(x) - \frac{x}{\alpha L} v(0) .
\end{align}
\end{subequations}
Setting $\mathcal{W}(t)=\left(w^1(t,\cdot),w^2(t,\cdot),w^3(t,\cdot)\right)$, $a(x) = (0,\frac{x}{\alpha L},0) \in\mathcal{H}^0$ and $b(x) = (0,0,-\frac{x}{\alpha L}) \in\mathcal{H}^0$, we have
\begin{subequations}\label{eq: cascade equation - homonegenous abstract}
\begin{align}
\dot{\mathcal{W}}(t) & = \mathcal{A} \mathcal{W}(t) + a v(t) + b \dot v(t) ,\\
\mathcal{W}(0,x) & = \left( y_0(x) , z_0(x) , z_1(x) - \dfrac{x}{\alpha L} v(0) \right) .
\end{align}
\end{subequations}
We now use the Riesz basis $\Phi$ and $\Psi$ defined in Section \ref{sec_prelim_A}, expanding
\begin{equation}\label{eq: projection system trajectory into Riesz basis}
\mathcal{W}(t,\cdot) = \sum_{n\in\mathbb{N}^*} w_{1,n}(t) \phi_{1,n} + \sum_{m \in\mathbb{Z}} w_{2,m}(t) \phi_{2,m} 
\end{equation}
with $w_{1,n} = \langle \mathcal{W}(t,\cdot) , \psi_{1,n} \rangle$, $w_{2,m} = \langle \mathcal{W}(t,\cdot) , \psi_{2,m} \rangle$, $a_{1,n} = \langle a , \psi_{1,n} \rangle$, $a_{2,m} = \langle a , \psi_{2,m} \rangle$, $b_{1,n} = \langle b , \psi_{1,n} \rangle$ and $b_{2,m} = \langle b , \psi_{2,m} \rangle$ for all $n\in\mathbb{N}^*$ and $m \in\mathbb{Z}$. 
Defining the new control $v_d = \dot{v}$ (auxiliary input for control design), we thus have
\begin{subequations}\label{eq: spectral reduction}
\begin{align}
\dot{w}_{1,n} & = \lambda_{1,n} w_{1,n} + a_{1,n} v + b_{1,n} v_d , \quad n\in\mathbb{N}^* , \label{eq: spectral reductiona} \\
\dot{w}_{2,m} & = \lambda_{2,m} w_{2,m} + a_{2,m} v + b_{2,m} v_d , \quad m \in\mathbb{Z}, \label{eq: spectral reductionb}
\end{align}
\end{subequations}
and
\begin{equation}\label{eq: inegral action}
\dot{v} = v_d .
\end{equation}
For a given $\delta >0$. let us choose an integer $N_0 \in\mathbb{N}^*$, large enough so that
$$
\lambda_{1,n} < -\delta < 0 \qquad\forall n \geq N_0 +1.
$$
We next consider the finite-dimensional system consisting of the $N_0$ first modes of the plant \eqref{eq: spectral reductiona} associated with the eigenvalues of the reaction-diffusion equation. Setting
\begin{align*}
W_0 & = \begin{pmatrix}
w_{1,1} & w_{1,2} & \ldots & w_{1,N_0}
\end{pmatrix}^\top \in\mathbb{R}^{N_0} , \\
A_0 & = \mathrm{diag}\left(
\lambda_{1,1} , \lambda_{1,2} , \ldots , \lambda_{1,N_0}
\right) \in\mathbb{R}^{N_0 \times N_0} , \\
B_{a,0} & = \begin{pmatrix}
a_{1,1} & a_{1,2} & \ldots & a_{1,N_0}
\end{pmatrix}^\top \in\mathbb{R}^{N_0} , \\
B_{b,0} & = \begin{pmatrix}
b_{1,1} & b_{1,2} & \ldots & b_{1,N_0}
\end{pmatrix}^\top \in\mathbb{R}^{N_0} ,
\end{align*}
we have 
\begin{equation}\label{truncature:eq}
\dot{W}_0(t) = A_0 W_0(t) + B_{a,0} v(t) + B_{b,0} v_d(t) .
\end{equation}
Augmenting the state vector and the matrices by setting
\begin{equation*}
W_1 = \begin{pmatrix} v \\ W_0 \end{pmatrix} , \quad
A_1 = \begin{pmatrix} 0 & 0 \\ B_{a,0} & A_0 \end{pmatrix} , \quad
B_1 = \begin{pmatrix} 1 \\ B_{b,0} \end{pmatrix} ,
\end{equation*}
we deduce that
\begin{equation}\label{truncature:eq:aug}
\dot{W}_1(t) = A_1 W_1(t) + B_1 v_d(t) .
\end{equation}
Since $v_d$ is viewed as an auxiliary control input, by Lemma~\ref{eq: necessary condition for observability inequality} and Theorem~\ref{thm: observability inequality}, the controllability of the pair $(A_1,B_1)$ is related to the coefficients $\gamma_k$. In particular, we have the following result (see also Remark \ref{rem_modes} further). 

\smallskip

\begin{lemma}\label{kalman:contro:lemma}
The pair $(A_1,B_1)$ satisfies the Kalman condition if and only if $\gamma_n \neq 0$ for any $n\in\{1,\ldots,N_0\}$, where $\gamma_n$ is defined by \eqref{eq: def gamma_1 - 1} if $\frac{n^2 \pi^2}{L^2}\neq c$ and by \eqref{eq: def gamma_1 - 2} if $\frac{n^2 \pi^2}{L^2} = c$.
\end{lemma}

\smallskip

\begin{proof}
By the Hautus test, $(A_1,B_1)$ does not satisfy the Kalman condition if and only if there exist $\lambda\in\mathbb{C}$ and $(x_1,x_2)\in\mathbb{C}\times\mathbb{C}^{N_0}\setminus\{(0,0)\}$ such that $\overline{x_2}^\top B_{a,0} = \lambda \overline{x_1}^\top$, $\overline{x_2}^\top A_0 = \lambda \overline{x_2}^\top$, and $\overline{x_1}^\top + \overline{x_2}^\top B_{b,0} = 0$. This is possible if and only if there exists $x_2 \neq 0$ such that $A_0^* x_2 = \overline{\lambda} x_2$ and $\overline{x_2}^\top \left( B_{a,0} + \lambda B_{b,0} \right) = 0$, i.e., if and only if $\lambda\in\Lambda=\{\lambda_{1,1} , \lambda_{1,2} , \ldots , \lambda_{1,N_0} \}$ while (because the eigenvalues of $\mathcal{A}^*$ are simple) there exists $z\in D(\mathcal{A}^*)\setminus\{0\}$ such that $\mathcal{A}^*z=\overline{\lambda}z$ and $\langle a + \lambda b , z \rangle = 0$. Setting $z=(z^1,z^2,z^3)$, the equation $\mathcal{A}^*z=\overline{\lambda}z$ gives
\begin{subequations}\label{eq: proof Kalman condition - sys eq to be solved - 1-5}
\begin{align}
& (z^1)'' + c z^1 = \overline{\lambda} z^1 \label{eq: proof Kalman condition - sys eq to be solved - 1} \\
& - z^3 + \int_0^{(\cdot)} \int_\tau^L \beta(s) z^1(s) \,\mathrm{d}s\,\mathrm{d}\tau = \overline{\lambda} z^2 \label{eq: proof Kalman condition - sys eq to be solved - 2} \\
& - (z^2)'' = \overline{\lambda} z^3 \label{eq: proof Kalman condition - sys eq to be solved - 3} \\
& z^1(0)=z^1(L)=z^2(0)=z^3(0)=0 \label{eq: proof Kalman condition - sys eq to be solved - 4} \\ 
& (z^2)'(L) - \alpha z^3(L) = 0 \label{eq: proof Kalman condition - sys eq to be solved - 5}
\end{align}
\end{subequations}
while 
\begin{multline*}
0 = \langle a + \lambda b , z \rangle 
= \int_0^L \left( \left(\frac{x}{\alpha L}\right)' \overline{(z^2)'(x)} - \frac{\lambda x}{\alpha L} \overline{z^3(x)} \right) \mathrm{d}x \\
 = \left[ \frac{x}{\alpha L} \overline{(z^2)'(x)} \right]_{x=0}^{x=L} - \int_0^L \frac{x}{\alpha L} \overline{\left( (z^2)''(x) + \overline{\lambda} z^3(x) \right)} \,\mathrm{d}x 
 = \frac{1}{\alpha} \overline{(z^2)'(L)}
\end{multline*}
where we have used \eqref{eq: proof Kalman condition - sys eq to be solved - 3}. Combining with \eqref{eq: proof Kalman condition - sys eq to be solved - 5}, we get
\begin{equation}\label{eq: proof Kalman condition - sys eq to be solved - 6}
(z^2)'(L) = z^3(L) = 0 .
\end{equation}

Let us discard the case $z^1 = 0$ by noting that, in this case, \eqref{eq: proof Kalman condition - sys eq to be solved - 2} gives $- z^3 = \overline{\lambda} z^2$. Combining this result with \eqref{eq: proof Kalman condition - sys eq to be solved - 3} and \eqref{eq: proof Kalman condition - sys eq to be solved - 6} we infer that $(z^3)'' - (\overline{\lambda})^2 z^3 = 0$ with $z^3(L)=(z^3)'(L)=0$. By Cauchy uniqueness, we deduce that $z^3 = 0$. Finally, \eqref{eq: proof Kalman condition - sys eq to be solved - 3} gives $(z^2)'' = 0$ with $z^2(0)=(z^2)'(L) = 0$, hence $z^2 =0$. This is a contradiction with the initial assumption that $z \neq 0$.

Hence, we must have $z^1 \neq 0$. Based on \eqref{eq: proof Kalman condition - sys eq to be solved - 1} and \eqref{eq: proof Kalman condition - sys eq to be solved - 4} we have $(z^1)'' = \overline{(\lambda-c)} z^1$ with $z^1(0)=z^1(L)=0$. Then there exist $A \neq 0$ and $n \in\mathbb{N}^*$ such that $z^1(x) = A\sin\left(\frac{n\pi}{L}x\right)$ and $\lambda = c - \frac{n^2 \pi^2}{L^2}$. Since $\lambda\in\Lambda$, we have $\lambda = \lambda_{1,n}$ for some $n\in\{1,\ldots,N_0\}$. We infer from \eqref{eq: proof Kalman condition - sys eq to be solved - 2}, \eqref{eq: proof Kalman condition - sys eq to be solved - 3}, and \eqref{eq: proof Kalman condition - sys eq to be solved - 6} that
\begin{subequations}
\begin{align}
& (z^3)''(x) - \lambda_{1,n}^2 z^3(x) = - A \beta(x) \sin\Big(\frac{n\pi}{L}x\Big) , \label{eq: proof Kalman condition - sys eq to be solved - 7} \\
& z^3(0) = z^3(L) = (z^3)'(L) = 0 , \label{eq: proof Kalman condition - sys eq to be solved - 8}
\end{align}
\end{subequations}
where we recall that $\lambda_{1,n} = c - \frac{n^2 \pi^2}{L^2}$.

Assume first that $\lambda_{1,n} \neq 0$. In this case, integrating \eqref{eq: proof Kalman condition - sys eq to be solved - 7}, there exists $\delta_1,\delta_2\in\mathbb{R}$ such that
\begin{multline*}
\!\!\!\!\! z^3(x) = \left( \delta_1 + \frac{A}{2\lambda_{1,n}} \int_0^x e^{\lambda_{1,n} s} \beta(s) \sin\Big( \frac{n \pi}{L} s \Big) \,\mathrm{d}s \right) e^{-\lambda_{1,n} x} \\
+ \left( \delta_2 - \frac{A}{2\lambda_{1,n}} \int_0^x e^{-\lambda_{1,n} s} \beta(s) \sin\Big( \frac{n \pi}{L} s \Big) \,\mathrm{d}s \right) e^{\lambda_{1,n} x} .
\end{multline*}
We now use \eqref{eq: proof Kalman condition - sys eq to be solved - 8}. The condition $z^3(0) = 0$ gives $\delta_1 + \delta_2 = 0$, hence 
\begin{equation*}
z^3(x) = - 2 \delta_1 \sinh(\lambda_{1,n} x) 
+ \frac{A}{\lambda_{1,n}} \int_0^x \beta(s)  \sinh(\lambda_{1,n}(s-x)) \sin\Big( \frac{n \pi}{L} s \Big) \,\mathrm{d}s .
\end{equation*}
The conditions $z^3(L) = 0$ and $(z^3)'(L) = 0$ give the system
\begin{equation*}
\begin{pmatrix}
\sinh(\lambda_{1,n} L) & \chi_1 \\
\cosh(\lambda_{1,n} L) & \chi_2
\end{pmatrix}
\begin{pmatrix}
-2 \delta_1 \lambda_{1,n} \\ A
\end{pmatrix}
= 0 
\end{equation*}
where $\chi_1 = \int_0^L \beta(s) \sinh(\lambda_{1,n}(s-L)) \sin\Big( \frac{n \pi}{L} s \Big) \,\mathrm{d}s$ and $\chi_2 = -\int_0^L \beta(s) \cosh(\lambda_{1,n}(s-L)) \sin\Big( \frac{n \pi}{L} s \Big) \,\mathrm{d}s$. Since $A \neq 0$ and $\lambda_{1,n} \neq 0$, the determinant of the above $2 \times 2$ matrix must be zero, i.e., 
\begin{align*}
0 = \int_0^L \beta(s) \sinh(\lambda_{1,n} s) \sin\Big( \frac{n \pi}{L} s \Big) \,\mathrm{d}s = \gamma_n .
\end{align*}
Under this condition, one can compute $A \neq 0$ and $\delta_1$, and thus obtain $z^3$. Finally, $z^2$ is obtained by integrating twice \eqref{eq: proof Kalman condition - sys eq to be solved - 3} and using the conditions $z^2(0)=(z^2)'(L)=0$ borrowed from \eqref{eq: proof Kalman condition - sys eq to be solved - 4} and \eqref{eq: proof Kalman condition - sys eq to be solved - 6}. The obtained $z=(z^1,z^2,z^3) \neq 0$ satisfies \eqref{eq: proof Kalman condition - sys eq to be solved - 1-5} and \eqref{eq: proof Kalman condition - sys eq to be solved - 6}.

Assume now that $\lambda_{1,n}=0$. In this case, \eqref{eq: proof Kalman condition - sys eq to be solved - 7} reduces to $(z^3)'' = - A \beta(x) \sin\left(\frac{n\pi}{L}x\right)$. Owing to \eqref{eq: proof Kalman condition - sys eq to be solved - 8}, we infer from the first and third conditions that $z^3(x)= A \int_0^x \int_\tau^L \beta(s) \sin\left(\frac{n\pi}{L}s\right) \,\mathrm{d}s\,\mathrm{d}\tau$ while the second condition gives 
$$0 = \int_0^L \int_\tau^L \beta(s) \sin\left(\frac{n\pi}{L}s\right) \,\mathrm{d}s\,\mathrm{d}\tau = \gamma_n$$
because $A \neq 0$. As previously, we then obtain $z^2$ by integrating twice \eqref{eq: proof Kalman condition - sys eq to be solved - 3} with the conditions $z^2(0)=(z^2)'(L)=0$ borrowed from \eqref{eq: proof Kalman condition - sys eq to be solved - 4} and \eqref{eq: proof Kalman condition - sys eq to be solved - 6}. Hence, the computed $z=(z^1,z^2,z^3) \neq 0$ satisfies \eqref{eq: proof Kalman condition - sys eq to be solved - 1-5} and \eqref{eq: proof Kalman condition - sys eq to be solved - 6}.
\end{proof}

\smallskip

\begin{remark}\label{rem_modes}
    While the exact controllability of the infinite-dimensional plant \eqref{eq: cascade equation} requires that $\gamma_n \neq 0$ for any $n\in\mathbb{N}^*$ (see Theorem~\ref{thm: observability inequality}), the controllability of the finite-dimensional dynamics \eqref{truncature:eq} capturing the $N_0$ first modes of the reaction-diffusion part of the PDE cascade only requires that $\gamma_n \neq 0$ for any $n\in\{1,\ldots,N_0\}$, i.e., only for the $N_0$ first modes. This  can be interpreted as follows. Since $\mathcal{X}(t)=\left(y(t,\cdot),z(t,\cdot),\partial_t z(t,\cdot)\right)$, we have $\mathcal{W} = \mathcal{X} + b v$. Hence
    \begin{subequations}\label{eq: spectral reduction bis}
    \begin{align}
    \dot{x}_{1,n} & = \lambda_{1,n} x_{1,n} + \beta_{1,n} v , \qquad n\in\mathbb{N}^*,\\
    \dot{x}_{2,m} & = \lambda_{2,m} x_{2,m} + \beta_{2,m} v , \qquad m \in\mathbb{Z},
    \end{align}
    \end{subequations}
    where $x_{1,n}(t) = \langle \mathcal{X}(t,\cdot) , \psi_{1,n} \rangle$, $x_{2,m}(t) = \langle \mathcal{X}(t,\cdot) , \psi_{2,m} \rangle$, 
    $\beta_{1,n} = a_{1,n} + \lambda_{1,n} b_{1,n}$ and $\beta_{2,m} = a_{2,m} + \lambda_{2,m} b_{2,m}$. In particular, the computations done in the proof of Lemma~\ref{kalman:contro:lemma} show that $\beta_{i,l} = \frac{1}{\alpha} \overline{(\psi_{i,l}^2)'(L) = \overline{\psi_{i,l}^3(L)}}$. This means that the control has no impact on the mode $\lambda_{i,l}$ if and only if $\psi_{i,l}^3(L) = 0$. This is in accordance with the observation based on \eqref{eq: observability inequality - system output} that led to Lemma~\ref{eq: necessary condition for observability inequality} and the conditions $\gamma_n \neq 0$. Since \eqref{truncature:eq} involves only a finite number of parabolic modes, this condition holds for a finite number of $\gamma_n$.
\end{remark}

\smallskip

In view of the control design, based on \eqref{eq: projection system trajectory into Riesz basis}, we note that the system output $y_o(t)$ defined by \eqref{eq: measurement heat distributed} is expanded as
\begin{multline*}
y_o(t) = \int_0^L c_o(x) y(t,x) \,\mathrm{d}x = \int_0^L c_o(x) w^1(t,x) \,\mathrm{d}x \\
 = \sum_{n\in\mathbb{N}^*} c_{1,n} w_{1,n}(t) + \sum_{m \in \mathbb{Z}} c_{2,m} w_{2,m}(t) 
= C_0 W_0(t) + \sum_{n \geq N_0 + 1} c_{1,n} w_{1,n}(t) + \sum_{m \in \mathbb{Z}} c_{2,m} w_{2,m}(t)
\end{multline*}
with 
\begin{subequations}\label{eq: def c_{i,k}}
\begin{align}
c_{1,n} & = \int_0^L c_o(x) \phi^1_{1,n}(x) \,\mathrm{d}x ,\quad n\in\mathbb{N}^*,  \\
c_{2,m} & = \int_0^L c_o(x) \phi^1_{2,m}(x) \,\mathrm{d}x ,\quad m \in \mathbb{Z}.
\end{align}
\end{subequations}
and
$$
C_0 = \begin{pmatrix}
c_{1,1} & c_{1,2} & \ldots & c_{1,N_0}
\end{pmatrix} \in\mathbb{R}^{1 \times N_0} .
$$
Since the matrix $A_0$ is diagonal with simple eigenvalues, we have the following result.

\smallskip

\begin{lemma}
The pair $(A_0,C_0)$ satisfies the Kalman condition if and only if 
$$
c_{1,n} =  \sqrt{\frac{2}{L}} \int_0^L \!\! c_o(x) \sin\Big( \frac{n\pi}{L} x \Big) \,\mathrm{d}x \neq 0 \quad \forall n\in\{1,\ldots,N_0\}. 
$$ 
\end{lemma}

\subsection{Output feedback control strategy and main result}\label{sec_main_feed}
\noindent
Given two integers $N \geq N_0 +1$ and $M \in\mathbb{N}$ to be chosen later, we define the following control strategy:
\begin{subequations}\label{eq: controller}
\begin{align}
\dot{\hat{w}}_{1,n} & = \lambda_{1,n} \hat{w}_{1,n} + a_{1,n} v + b_{1,n} v_d  - l_{1,n} \bigg( \sum_{k=1}^N c_{1,k} \hat{w}_{1,k} + \sum_{\vert l \vert \leq M} c_{2,l} \hat{w}_{2,l} - y_o \bigg) , \quad 1 \leq n \leq N_0 , \\
\dot{\hat{w}}_{1,n} & = \lambda_{1,n} \hat{w}_{1,n} + a_{1,n} v + b_{1,n} v_d ,  \quad N_0 + 1 \leq n \leq N , \\
\dot{\hat{w}}_{2,m} & = \lambda_{2,m} \hat{w}_{2,m} + a_{2,m} v + b_{2,m} v_d , \quad \vert m \vert \leq M , \\
v_d & = k_v v + \sum_{n=1}^{N_0} k_{1,n} \hat{w}_{1,n} ,
\end{align}
\end{subequations}
where $k_v$ and $k_{1,n}$, for $1\leq n\leq N_0$, are the feedback gains and $l_{1,n}$, for $1\leq n\leq N_0$, are the observer gains. This finite-dimensional control strategy leveraging a Luenberger-type observer on a finite number of modes is inspired by the seminal works~\cite{sakawa1983feedback} and its more recent developments~\cite{katz2020constructive, lhachemi2020finite, lhachemi2021nonlinear, grune2021finite}. We set
\begin{align*}
K & = \begin{pmatrix}
k_v & k_{1,1} & k_{1,2} & \ldots & k_{1,N_0}
\end{pmatrix} \in\mathbb{R}^{1 \times (N_0+1)} , \qquad
L  = \begin{pmatrix}
l_{1,1} & l_{1,2} & \ldots & l_{1,N_0}
\end{pmatrix}^\top \in\mathbb{R}^{N_0} .
\end{align*}


\smallskip

\begin{theorem}\label{thm: main result 2}
Let $\delta > 0$ be arbitrary. Let $N_0\in\mathbb{N}^*$ and $\alpha > 1$ be such that $\lambda_{1,N_0+1} < - \delta$ and $\rho = \frac{1}{2L} \log\big(\frac{\alpha-1}{\alpha+1}\big) = \operatorname{Re}\lambda_{2,m} < - \delta$. Assume that:
\begin{itemize}
\item $\gamma_n \neq 0$ for any $n\in\{1,\ldots,N_0\}$, where $\gamma_n$ is defined by \eqref{eq: def gamma_1 - 1} if $\frac{n^2 \pi^2}{L^2}\neq c$ and by \eqref{eq: def gamma_1 - 2} if $\frac{n^2 \pi^2}{L^2} = c$;
\item $c_{1,n} \neq 0$ for any $n\in\{1,\ldots,N_0\}$, where $c_{1,n}$ is defined by \eqref{eq: def c_{i,k}}. 
\end{itemize}
Let $K\in\mathbb{R}^{1 \times (N_0 +1)}$ and $L\in\mathbb{R}^{N_0}$ be such that $A_1 + B_1 K$ and $A_0 - L C_0$ are Hurwitz with eigenvalues of real part less than $-\delta < 0$. 

Then, for all integers $N \geq N_0 +1$ and $M$ sufficiently large\footnote{They must be chosen large enough so that the inequalities \eqref{constraints_NM} are satisfied, see the proof.},
there exists $C > 0$ such that 
any solution of the system \eqref{eq: cascade equation} in closed-loop with the output feedback control \eqref{eq: measurement heat distributed}, \eqref{eq: preliminary feedback}, \eqref{eq: inegral action}, \eqref{eq: controller} satisfies
\begin{align}
& \big\Vert ( y(t,\cdot) , z(t,\cdot) , \partial_t z(t,\cdot) ) \big\Vert_{Y} + \vert v(t) \vert  + \sum_{n=1}^N \vert \hat{w}_{1,n}(t) \vert + \sum_{\vert m \vert \leq M} \vert \hat{w}_{2,m}(t) \vert \nonumber \\
& \leq
C e^{-\delta t}
\Big(
\big\Vert ( y(0,\cdot) , z(0,\cdot) , \partial_t z(0,\cdot) ) \big\Vert_{Y} + \vert v(0) \vert  + \sum_{n=1}^N \vert \hat{w}_{1,n}(0) \vert + \sum_{\vert m \vert \leq M} \vert \hat{w}_{2,m}(0) \vert \Big)  \label{eq: exponential stability}
\end{align}
for every $t\geq 0$, where $Y$ is either the Hilbert space $\mathcal{H}^0$ defined by \eqref{eq: state-space},
or $\mathcal{H}^1$ defined by \eqref{eq: space H1}. 
\end{theorem}

\smallskip

\begin{remark}
In the above statement, it is understood that, if one takes an initial condition $(y(0,\cdot) , z(0,\cdot) , \partial_t z(0,\cdot))$ in $Y$, then there is a unique solution, remaining in $Y$. This fact follows from the proof.
\end{remark}

\smallskip

\begin{proof}
We define the observation discrepancies $e_{1,n} = w_{1,n} - \hat{w}_{1,n}$ and $e_{2,m} = w_{2,m} - \hat{w}_{2,m}$ as well as the vectors
\begin{align*}
\hat{W}_0 & = \begin{pmatrix}
\hat{w}_{1,1} & \hat{w}_{1,2} & \ldots & \hat{w}_{1,N_0} 
\end{pmatrix}^\top , \\
E_1 & = \begin{pmatrix}
e_{1,1} & e_{1,2} & \ldots & e_{1,N_0} 
\end{pmatrix}^\top , \\
\hat{W}_2 & = \begin{pmatrix}
\hat{w}_{1,N_0 + 1} & \ldots & \hat{w}_{1,N} & \hat{w}_{2,0} & \ldots & \hat{w}_{2,-M} & \hat{w}_{2,M}
\end{pmatrix}^\top , \\
E_2 & = \begin{pmatrix}
e_{1,N_0 + 1} & \ldots & e_{1,N} & e_{2,0} & \ldots & e_{2,-M} & e_{2,M}
\end{pmatrix}^\top .
\end{align*}
We infer from \eqref{eq: spectral reduction} and \eqref{eq: controller} that
\begin{subequations}
\begin{align}
\dot{\hat{W}}_0 & = A_0 \hat{W}_0 + B_{a,0} v + B_{b,0} v_d + L C_0 E_1 + L C_1 E_2 + L \zeta_1 + L \zeta_2 , \\
\dot{E}_1 & = ( A_0 - L C_0 ) E_1 - L C_1 E_2 - L \zeta_1 - L \zeta_2 , \\
\dot{\hat{W}}_2 & = A_2 \hat{W}_2 + B_{a,1} v + B_{b,1} v_d , \\
\dot{E}_2 & = A_2 E_2 ,
\end{align}
\end{subequations} 
where 
\begin{align*}
A_2 & = \mathrm{diag}\left(
\lambda_{1,N_0 + 1} , \ldots , \lambda_{1,N} , \lambda_{2,0} , \ldots , \lambda_{2,-M} , \lambda_{2,M}
\right) , \\
B_{a,1} & = \begin{pmatrix}
a_{1,N_0+1} & \ldots & a_{1,N} & a_{2,0} & \ldots & a_{2,-M} & a_{2,M}
\end{pmatrix}^\top , \\
B_{b,1} & = \begin{pmatrix}
b_{1,N_0+1} & \ldots & b_{1,N} & b_{2,0} & \ldots & b_{2,-M} & b_{2,M}
\end{pmatrix}^\top , \\
C_1 & = \begin{pmatrix}
c_{1,N_0+1} & \ldots & c_{1,N} & c_{2,0} & \ldots & c_{2,-M} & c_{2,M}
\end{pmatrix} , \\
\zeta_1 & = \sum_{k\geq N+1} c_{1,k} w_{1,k} , \quad
\zeta_2 = \sum_{\vert l \vert \geq M+1} c_{2,l} w_{2,l} ,
\end{align*}
with $A_2 \in\mathbb{C}^{N-N_0+2M+1}$, $B_{a,1},B_{b,1} \in\mathbb{C}^{N-N_0+2M+1}$ and $C_1 \in\mathbb{C}^{N-N_0+2M+1}$.
Defining the augmented vectors
\begin{equation*}
\hat{W}_1 = \begin{pmatrix} v \\ \hat{W}_0 \end{pmatrix}
, \quad
\tilde{L} = \begin{pmatrix} 0 \\ L \end{pmatrix} ,
\end{equation*}
we obtain
\begin{subequations} 
\begin{align}
v_d & = K \hat{W}_1 , \\
\dot{\hat{W}}_1 & = (A_1+B_1 K) \hat{W}_1 + \tilde{L} C_0 E_1 + \tilde{L} C_1 E_2 + \tilde{L} \zeta_1 + \tilde{L} \zeta_2 .
\end{align}
\end{subequations} 
Finally, setting
\begin{equation}
X = \mathrm{col}\big( \hat{W}_1 , E_1 , \hat{W}_2 , E_2 \big) ,
\end{equation}
we have
\begin{equation}
\dot{X} = F X + \mathcal{L} \zeta_1 + \mathcal{L} \zeta_2
\end{equation}
where 
\begin{subequations}\label{eq: def matrix F}
\begin{align}
&F =  \begin{pmatrix}
A_1 + B_1 K & \tilde{L} C_0 & 0 & \tilde{L} C_1 \\
0 & A_0 - L C_0 & 0 & - L C_1 \\
B_{a,1} \Big( 1 \ 0 \ \cdots \ 0 \Big) + B_{b,1} K & 0 & A_2 & 0 \\
0 & 0 & 0 & A_2
\end{pmatrix} , \\
& \mathcal{L} = \begin{pmatrix}
\tilde{L}^\top & - L^\top & 0 & 0
\end{pmatrix}^\top .
\end{align}
\end{subequations}
Let us now define a suitable Lyapunov functional. Let $P$ be a Hermitian positive definite matrix (to be chosen later). 

For the parabolic part of the system evaluated in $L^2$ norm (i.e., for the state of the system evaluated in $\mathcal{H}^0$ norm), we define
\begin{equation}\label{eq: V L2 norm}
V(X,w) = \bar{X}^\top P X + \sum_{n \geq N+1}\!  \vert w_{1,n} \vert^2 + \sum_{\vert m \vert \geq M+1}\!  \vert w_{2,m} \vert^2 .
\end{equation}
Since $\Phi$ is a Riesz basis of $\mathcal{H}^0$ (see Lemma~\ref{lem: A riesz spectral}), of dual basis $\Psi$ (see Lemma~\ref{lem: dual basis}), it follows that $\sqrt{V}$ is a norm, equivalent to the norm of $\mathbb{C}^{2N+1}\times\mathcal{H}^0$. 

For the parabolic part of the system evaluated in $H^1$ norm (i.e., for the state of the system evaluated in $\mathcal{H}^1$ norm), we define
\begin{equation}\label{eq: V H1 norm}
V(X,w) = \bar{X}^\top P X + \sum_{n \geq N+1} n^2 \vert w_{1,n} \vert^2 + \sum_{\vert m \vert \geq M+1} \vert w_{2,m} \vert^2 .
\end{equation}
Since $\Phi^1$ is a Riesz basis of $\mathcal{H}^1$ (see Lemma~\ref{lemma: Riesz basis H1}), $\sqrt{V}$ is a norm, equivalent to the norm of $\mathbb{C}^{2N+1}\times\mathcal{H}^1$. Indeed, using \eqref{eq: projection system trajectory into Riesz basis}, we have
\begin{align*}
W(t,\cdot) & = \left(w^1(t,\cdot),w^2(t,\cdot),w^3(t,\cdot)\right) \nonumber \\
& = \sum_{n\in\mathbb{N}^*} w_{1,n}(t) \phi_{1,n} + \sum_{m \in\mathbb{Z}} w_{2,m}(t) \phi_{2,m} \\
& = \sum_{n\in\mathbb{N}^*} \dfrac{n \pi}{L} w_{1,n}(t) \dfrac{L}{n \pi} \phi_{1,n} + \sum_{m \in\mathbb{Z}} w_{2,m}(t) \phi_{2,m} .
\end{align*}
The above two series converge \emph{a priori} in $\mathcal{H}^0$ norm. However, for a classical solution $\mathcal{W}(t,\cdot)\in D(\mathcal{A})$, using Lemma~\ref{lemma: Riesz basis H1}, the series converge in $\mathcal{H}^1$ norm. Hence, thanks to Lemma~\ref{lemma: Riesz basis H1}, $\Vert \mathcal{W}(t,\cdot) \Vert_{\mathcal{H}^1}^2$ is equivalent to $\sum_{n\in\mathbb{N}^*} n^2 \vert w_{1,n}(t) \vert^2 + \sum_{m \in\mathbb{Z}} \vert w_{2,m}(t) \vert^2$. This justifies the definition of \eqref{eq: V H1 norm}.

Since the proofs in $\mathcal{H}^0$ norm and in $\mathcal{H}^1$ norm are now similar, we focus on the second case. Setting $\tilde{X} = \mathrm{col}(X,\zeta_1,\zeta_2)$, the computation of the time derivative of $V$ along the system trajectories gives
\begin{align*}
& \dot{V} = \bar{\tilde{X}}^\top 
\begin{pmatrix} 
\bar{F}^\top P + P F & P \mathcal{L} & P \mathcal{L} \\
\mathcal{L}^\top P & 0 & 0 \\
\mathcal{L}^\top P & 0 & 0
\end{pmatrix} 
\tilde{X} 
 + 2 \sum_{n \geq N+1} n^2 \operatorname{Re} \left( \left( \lambda_{1,n} w_{1,n} + a_{1,n} v + b_{1,n} v_d \right) \overline{w_{1,n}} \right) \\
& \qquad\qquad\qquad + 2 \sum_{\vert m \vert \geq M+1} \operatorname{Re} \left( \left( \lambda_{2,m} w_{2,m} + a_{2,m} v + b_{2,m} v_d \right) \overline{w_{2,m}} \right) \\
& \leq \bar{\tilde{X}}^\top  
\begin{pmatrix} 
\bar{F}^\top P + P F & P \mathcal{L} & P \mathcal{L} \\
\mathcal{L}^\top P & 0 & 0 \\
\mathcal{L}^\top P & 0 & 0
\end{pmatrix} 
\tilde{X} \\
& \phantom{\leq}\, + 2 \sum_{n \geq N+1} n^2 \lambda_{1,n} \vert w_{1,n} \vert^2
+ 2 \sum_{\vert m \vert \geq M+1} \underbrace{\operatorname{Re}\lambda_{2,m}}_{= \rho} \vert w_{2,m} \vert^2  + \frac{2}{\epsilon} \bigg( \sum_{n \geq N+1} n^4 \vert w_{1,n} \vert^2 + \sum_{\vert m \vert \geq M+1} \vert w_{2,m} \vert^2 \bigg) \\
& \phantom{\leq}\, + \epsilon \bigg( \sum_{n \geq N+1} \vert a_{1,n} \vert^2 + \sum_{\vert m \vert \geq M+1} \vert a_{2,m} \vert^2 \bigg) \vert v \vert^2 
 + \epsilon \bigg( \sum_{n \geq N+1} \vert b_{1,n} \vert^2 + \sum_{\vert m \vert \geq M+1} \vert b_{2,m} \vert^2 \bigg) \vert v_d \vert^2 
\end{align*}
with $\epsilon > 0$ arbitrary and $\rho = \frac{1}{2L} \log\big(\frac{\alpha-1}{\alpha+1}\big) = \operatorname{Re}\lambda_{2,m}$. The latter inequality has been obtained by using Young's inequality $ab\leq\frac{a^2}{2\epsilon}+\frac{\epsilon b^2}{2}$. Defining now $E = \begin{pmatrix} 1 & 0 & \ldots & 0 \end{pmatrix}$ and $\tilde{K} = \begin{pmatrix} K & 0 & 0 & 0 \end{pmatrix}$, we infer that $v = E X$ and $v_d = \tilde{K} X$. Furthermore, using the Cauchy-Schwarz inequality,
\begin{align*}
\zeta_1^2 & = \bigg( \sum_{k\geq N+1} c_{1,k} w_{1,k} \bigg)^2 \leq \underbrace{\sum_{k\geq N+1} \vert c_{1,k} \vert^2}_{= S_{c,1,N}} \times \sum_{k\geq N+1} \vert w_{1,k} \vert^2 , \\
\zeta_2^2 & = \bigg( \sum_{\vert l \vert\geq M+1} c_{2,l} w_{2,l} \bigg)^2 \leq \underbrace{\sum_{\vert l \vert\geq M+1} \vert c_{2,l} \vert^2}_{= S_{c,2,M}} \times \sum_{\vert l \vert\geq M+1} \vert w_{2,l} \vert^2 .
\end{align*}
Defining $S_{a,N,M} = \sum_{n \geq N+1} \vert a_{1,n} \vert^2 + \sum_{\vert m \vert \geq M+1} \vert a_{2,m} \vert^2$ and $S_{b,N,M} = \sum_{n \geq N+1} \vert b_{1,n} \vert^2 + \sum_{\vert m \vert \geq M+1} \vert b_{2,m} \vert^2$, the combination of all above estimates gives 
\begin{equation}\label{eq: dotV}
\dot{V} + 2 \delta V \leq \bar{\tilde{X}}^\top \Theta \tilde{X} 
+ \sum_{n \geq N+1} n^2 \Gamma_{1,n} \vert w_{1,n} \vert^2  
+ \Gamma_{2,M} \sum_{\vert m \vert \geq M+1} \vert w_{2,m} \vert^2 
\end{equation}
with
\begin{subequations}
\begin{align}
\Theta & =
\begin{pmatrix} 
\Theta_{1,1} & P \mathcal{L} & P \mathcal{L} \\
\mathcal{L}^\top P & -\eta_1 & 0 \\
\mathcal{L}^\top P & 0 & -\eta_2
\end{pmatrix} , \label{eq: stab condition - Theta} \\
\Gamma_{1,n} & = 2 \Big( \lambda_{1,n} + \frac{n^2}{\epsilon} + \delta \Big) + \frac{\eta_1 S_{c,1,N}}{n^2} , \\
\Gamma_{2,M} & = 2 \Big( \rho + \frac{1}{\epsilon} + \delta \Big) + \eta_2 S_{c,2,M} ,  \label{eq: stab condition - Gamma2m}
\end{align}
\end{subequations}
where $\Theta_{1,1} = \bar{F}^\top P + P F + 2\delta P + \epsilon S_{a,N,M} E^\top E + \epsilon S_{b,N,M} \tilde{K}^\top \tilde{K}$ and $\eta_1,\eta_2 > 0$ arbitrary. Choosing $\epsilon > L^2/\pi^2$, we have $\Gamma_{1,n} = 2 \Big( - \left( \frac{\pi^2}{L^2} - \frac{1}{\epsilon} \right) n^2 + c + \delta \Big) + \frac{\eta_1 S_{c,1,N}}{n^2} \leq \Gamma_{1,N+1}$ for any $n \geq N+1$. Hence, $\dot{V} + 2 \delta V \leq 0$ (yielding the claimed stability estimate \eqref{eq: exponential stability}) provided there exist integers $N \geq N_0 + 1$ and $M \in\mathbb{N}$, real numbers $\epsilon > L^2/\pi^2$ and $\eta_1,\eta_2 > 0$, and a Hermitian matrix $P \succ 0$ such that
\begin{equation}\label{constraints_NM}
\Theta \preceq 0 , \quad \Gamma_{1,N+1} \leq 0 , \quad \Gamma_{2,M} \leq 0 .
\end{equation}
To conclude the proof, it thus remains to prove that the constraints \eqref{constraints_NM} are always feasible for $N,M$ chosen large enough. Recalling that the matrix $F$ is defined by \eqref{eq: def matrix F}, it is easy to see that $F$ is Hurwitz with eigenvalues of real part less than $-\delta < 0$, while the application of \cite[Appendix]{lhachemi2020finite} shows that the solution $P \succ 0$ of the Lyapunov equation $\bar{F}^\top P + PF + 2\delta P = -I$ is such that $\Vert P \Vert = O(1)$ as $N,M \rightarrow + \infty$. Furthermore, $\Vert E \Vert$, $\Vert \mathcal{L} \Vert$, and $\Vert \tilde{K} \Vert$ are constants, not depending on $N,M$ while $S_{a,N,M},S_{b,N,M},S_{c,1,N},S_{c,2,M} \rightarrow 0$ as $N,M \rightarrow + \infty$. Recalling that, by assumption, $\rho < - \delta$, we take $\epsilon > L^2/\pi^2$ large enough so that $\rho + \frac{1}{\epsilon} + \delta < 0$. We finally take
\begin{equation*}
\eta_1 = \left\{\begin{array}{cl}
\frac{1}{\sqrt{S_{c,1,N}}} & \mathrm{if}\; S_{c,1,N} \neq 0, \\
N & \textrm{otherwise},
\end{array}\right.
\qquad
\eta_2 = \left\{\begin{array}{cl}
\frac{1}{\sqrt{S_{c,2,M}}} & \mathrm{if}\; S_{c,2,M} \neq 0, \\
M & \textrm{otherwise},
\end{array}\right.
\end{equation*}
which, in particular, implies that $\eta_1,\eta_2 \rightarrow +\infty$ while $\eta_1 S_{c,1,N} , \eta_2 S_{c,2,M} \rightarrow 0$ as $N,M \rightarrow +\infty$. With these choices, it can be seen that $\Gamma_{1,N+1} \rightarrow -\infty$ as $N \rightarrow +\infty$ and $\Gamma_{2,M} \rightarrow 2 \left( \rho + \frac{1}{\epsilon} + \delta \right) < 0$ as $M \rightarrow +\infty$. Hence $\Gamma_{1,N+1} \leq 0$ and $\Gamma_{2,M} \leq 0$ for all $N,M$ large enough. Finally, the Schur complement theorem applied to
\begin{equation*}
\Theta =
\begin{pmatrix} 
-I + \epsilon S_{a,N,M} \bar{E}^\top E + \epsilon S_{b,N,M} \tilde{K}^\top \tilde{K} & P \mathcal{L} & P \mathcal{L} \\
\mathcal{L}^\top P & -\eta_1 & 0 \\
\mathcal{L}^\top P & 0 & -\eta_2
\end{pmatrix} .
\end{equation*}
shows that $\Theta \preceq 0$ if and only if 
\begin{equation*}
-I + \epsilon S_{a,N,M} E^\top E + \epsilon S_{b,N,M} \tilde{K}^\top \tilde{K} 
 + \begin{pmatrix} P \mathcal{L} & P \mathcal{L} \end{pmatrix}
\begin{pmatrix} \frac{1}{\eta_1} & 0 \\
0 & \frac{1}{\eta_2} \end{pmatrix} 
\begin{pmatrix} \mathcal{L}^\top P \\ \mathcal{L}^\top P \end{pmatrix}
\preceq 0 .
\end{equation*}
We note that 
\begin{equation*}
-I + \epsilon S_{a,N,M} E^\top E + \epsilon S_{b,N,M} \tilde{K}^\top \tilde{K} 
\preceq - \left( 1 - \epsilon \left( S_{a,N,M} \Vert E \Vert^2 +  S_{b,N,M} \Vert \tilde{K} \Vert^2 \right) \right) I 
\preceq - \frac{1}{2} I
\end{equation*}
for $N,M$ taken large enough, because $\epsilon ( S_{a,N,M} \Vert E \Vert^2 +  S_{b,N,M} \Vert \tilde{K} \Vert^2 ) \rightarrow 0$ as $N,M \rightarrow + \infty$. In this case,
\begin{multline*}
-I + \epsilon S_{a,N,M} E^\top E + \epsilon S_{b,N,M} \tilde{K}^\top \tilde{K} 
+ \begin{pmatrix} P \mathcal{L} & P \mathcal{L} \end{pmatrix}
\begin{pmatrix} \frac{1}{\eta_1} & 0 \\
0 & \frac{1}{\eta_2} \end{pmatrix} 
\begin{pmatrix} \mathcal{L}^\top P \\ \mathcal{L}^\top P \end{pmatrix} \\
\preceq
- \frac{1}{2} I + \left( \frac{1}{\eta_1} + \frac{1}{\eta_2} \right) P \mathcal{L} \mathcal{L}^\top P .
\end{multline*}
Recalling that $\eta_1,\eta_2 \rightarrow +\infty$ and $\Vert P \Vert = O(1)$ as $N,M \rightarrow +\infty$ while $\Vert \mathcal{L} \Vert$ is a constant, not depending on $N,M$, the latter matrix is $\preceq 0$ for $N,M$ chosen large enough. This completes the proof.
\end{proof}

\section{Extension to a pointwise Dirichlet or Neumann measurement for the heat equation}\label{sec: extemsion}
\noindent
In this section, instead of the distributed measurement \eqref{eq: measurement heat distributed} of the reaction-diffusion PDE, we consider the case of either the Dirichlet trace
\begin{equation}\label{eq: Dirichlet measurement}
	y_o(t) = y(t,\xi_{p})
\end{equation}
or the Neumann trace
\begin{equation}\label{eq: Naumann measurement}
	y_o(t) = \partial_x y(t,\xi_{p})
\end{equation}
for a fixed $\xi_{p}\in[0,L]$. We infer from \eqref{eq: change of variable} and \eqref{eq: projection system trajectory into Riesz basis} that either
$$
	y_o(t)  = w(t,\xi_{p}) 
	 = \sum_{n\in\mathbb{N}^*} w_{1,n}(t) \underbrace{\phi_{1,n}^1(\xi_{p})}_{=c_{1,n}} + \sum_{m \in\mathbb{Z}} w_{2,m}(t) \underbrace{\phi_{2,m}^1(\xi_{p})}_{=c_{2,m}}
$$
or
$$
	y_o(t) = \partial_x w(t,\xi_{p}) 
	 = \sum_{n\in\mathbb{N}^*} w_{1,n}(t) \underbrace{(\phi_{1,n}^1)'(\xi_{p})}_{=c_{1,n}} + \sum_{m \in\mathbb{Z}} w_{2,m}(t) \underbrace{(\phi_{2,m}^1)'(\xi_{p})}_{=c_{2,m}}. 
$$

\begin{theorem}
Considering either the Dirichlet measurement \eqref{eq: Dirichlet measurement} or Neumann measurement \eqref{eq: Naumann measurement} for some $\xi_{p}\in[0,L]$, the same statement as Theorem \ref{thm: main result 2} holds true for the Hilbert space $Y = \mathcal{H}^1$ defined by \eqref{eq: space H1}. 
\end{theorem}

\smallskip

\begin{proof}
The assumption that $c_{1,n} \neq 0$ for any $n\in\{1,\ldots,N_0\}$ implies the controllability of the pair $(A_0,C_0)$. Now, the proof follows the one of Theorem~\ref{thm: main result 2} by using the Lyapunov functional \eqref{eq: V H1 norm}. The only difference lies in the estimates of $\zeta_1 = \sum_{n\geq N+1} c_{1,n} w_{1,n}$ and of $\zeta_2 = \sum_{\vert m \vert \geq M+1} c_{2,m}^1(\xi_{p}) w_{2,m}$. 
	
In the case of the Dirichlet measurement, by Lemma~\ref{lem: eigenstructures of A}, we have $\phi_{1,n}^1(\xi_{p}) = O(1)$, hence 
\begin{equation*}
\zeta_1^2 \leq \underbrace{\sum_{n\geq N+1} \frac{\vert \phi_{1,n}^1(\xi_{p}) \vert^2}{n^2}}_{= S_{c,1,N} < +\infty} \times \sum_{n\geq N+1} n^2 \vert w_{1,n} \vert^2 .
\end{equation*}
Moreover, we showed in the proof of Lemma~\ref{lem: A riesz spectral} that $\Vert \phi_{2,m}^1 \Vert_{L^\infty} = O(1/m^2)$, yielding
\begin{equation*}
\zeta_2^2 \leq \underbrace{\sum_{\vert m \vert\geq M+1} \vert \phi_{2,m}^1(\xi_p) \vert^2}_{= S_{c,2,M} < +\infty} \times \sum_{\vert m \vert\geq M+1} \vert w_{2,m} \vert^2 .
\end{equation*}
Then, proceeding as in the proof of Theorem~\ref{thm: main result 2}, we obtain \eqref{eq: dotV} with 
\begin{equation*}
\Gamma_{1,n} = 2 \Big( \lambda_{1,n} + \frac{n^2}{\epsilon} + \delta \Big) + \eta_1 S_{c,1,N} ,
\end{equation*}
while $\Theta$ and $\Gamma_{2,M}$ are defined by \eqref{eq: stab condition - Theta} and \eqref{eq: stab condition - Gamma2m}, respectively. In order to apply \cite[Appendix]{lhachemi2020finite} to ensure the feasibility of the design constraints, the components of $E^2$ have to be rescaled in order to ensure that $\Vert C_1 \Vert = O(1)$ as $N \rightarrow +\infty$ (see \cite{lhachemi2020finite}). The proof is then similar to the one of Theorem~\ref{thm: main result 2}.

In the case of the Neumann measurement, by Lemma~\ref{lem: eigenstructures of A}, we have $(\phi_{1,n}^1)'(\xi_{p}) = O(n)$, hence 
\begin{equation*}
\zeta_1^2 \leq \underbrace{\sum_{n\geq N+1} \frac{\vert (\phi_{1,n}^1)'(\xi_{p}) \vert^2}{n^{7/2}}}_{= S_{c,1,N} < \infty} \times \sum_{n\geq N+1} n^{7/2} \vert w_{1,n} \vert^2 .
\end{equation*}
Moreover, following the proof in Lemma~\ref{lem: A riesz spectral} that $\Vert (\phi_{2,m}^1)' \Vert_{L^\infty} = O(1/\vert m \vert^{3/2})$, we have
\begin{equation*}
\zeta_2^2 \leq \underbrace{\sum_{\vert m \vert\geq M+1} \vert \phi_{2,m}^1(\xi_p) \vert^2}_{= S_{c,2,M} < \infty} \times \sum_{\vert m \vert\geq M+1} \vert w_{2,m} \vert^2 .
\end{equation*}
Then, proceeding as in the proof of Theorem~\ref{thm: main result 2}, we obtain \eqref{eq: dotV} with 
\begin{equation*}
\Gamma_{1,n} = 2 \Big( \lambda_{1,n} + \frac{n^2}{\epsilon} + \delta \Big) + n^{3/2} \eta_1 S_{c,1,N} ,
\end{equation*}
while $\Theta$ and $\Gamma_{2,M}$ are defined by \eqref{eq: stab condition - Theta} and \eqref{eq: stab condition - Gamma2m}, respectively. The proof is then similar to the one of Theorem~\ref{thm: main result 2}.
\end{proof}

\section{Conclusion}
\noindent
We have studied controllability and stabilization properties for the wave-reaction-diffusion PDE cascade \eqref{eq: cascade equation}, where the solution of the wave equation appears as a source term in the heat equation. The exact controllability Hilbert space has been characterized thanks to an Ingham-M{\"u}ntz type inequality. An explicit output feedback control strategy has then been derived. 

Although the system \eqref{eq: cascade equation} is a simple model, it already reveals an interesting complexity of the exact controllability space, depending on the coupling function, and it  may serve as a reference to study more elaborate models.

For instance, it is natural to consider the ``symmetric'' cascade model where the solution of the heat equation appears as a source term in the wave equation. While the study of the exact controllability of the system seems to be achievable using the same approach, the actual design of an explicit feedback control strategy seems much more challenging because the preliminary shift of the spectrum of the wave equation achieved by \eqref{eq: preliminary feedback} cannot be applied anymore. This is left as an open question for future research. 

Throughout the article, we have assumed $c$ constant in \eqref{eq: cascade equation - 1}. Actually, it is not difficult to generalize all results of that paper to the case $c\in L^\infty(0,L)$, but this complicates the spectral analysis of the operators. For large values of $n$ and $m$, the eigenvalues and eigenfunctions keep, asymptotically, the form given in Lemmas \ref{lem: eigenstructures of A0} and \ref{lem: adjoint A0*}, but for low frequencies their expression may be very different, thus having an impact on the coefficients $\gamma_n$ defined by \eqref{def_gamma_n}. Then, more technical assumptions appear to characterize controllability.

Another question is to explore controllability and stabilization properties for multi-dimensional versions of \eqref{eq: cascade equation}, in a general domain. The study is expected to be much more challenging because the formalism of Riesz basis cannot be used.

\appendix
\section{Appendix}
\subsection{Defining $\mathcal{B}_0$ by transposition}\label{sec_app_A}
\noindent
At the beginning of Section \ref{sec: exact controllability}, following the theory of well-posed linear systems developed in \cite{tucsnakweiss} (see also \cite{trelat_SB}), we have written the control system \eqref{eq: cascade equation} in the abstract form
\begin{equation}\label{abstract_system}
\dot{\mathcal{X}}(t)=\mathcal{A}_0 \mathcal{X}(t)+\mathcal{B}_0 u(t)
\end{equation}
with $\mathcal{A}_0:D(\mathcal{A}_0)\rightarrow\mathcal{H}^0$ defined by \eqref{def_A0}, and the Hilbert space $\mathcal{H}^0$ is defined by \eqref{eq: state-space}.
In this appendix, we show how to define the control operator $\mathcal{B}_0$ by transposition. 
Recall that the adjoint operator $\mathcal{A}_0^*:D(\mathcal{A}_0^*)\rightarrow\mathcal{H}^0$ is defined by \eqref{def_A0*}.
Following \cite{trelat_SB, tucsnakweiss}, identifying $\mathcal{H}^0$ with its dual, we search $\mathcal{B}_0\in L(\mathbb{R},D(\mathcal{A}_0^*)')$, or equivalently, $\mathcal{B}_0^*\in L(D(\mathcal{A}_0^*),\mathbb{R})$, where $D(\mathcal{A}_0^*)'$ is the dual of $D(\mathcal{A}_0^*)$ with respect to the pivot space $\mathcal{H}^0$.
Note that $D(\mathcal{A}_0^*) \subset \mathcal{H}^0\subset D(\mathcal{A}_0^*)'$ with continuous and dense embeddings.
Since the equality $\dot{\mathcal{X}}=\mathcal{A}_0 \mathcal{X}+\mathcal{B}_0 u$ is written in the space $D(\mathcal{A}_0^*)'$, using the duality bracket $\langle\ ,\ \rangle_{D(\mathcal{A}_0^*)',D(\mathcal{A}_0^*)}$, we have, for any $\phi\in D(\mathcal{A}_0^*)$, 
\begin{equation*}
\langle\dot{\mathcal{X}}(t),\phi\rangle_{D(\mathcal{A}_0^*)',D(\mathcal{A}_0^*)} 
= \langle\mathcal{A}_0 \mathcal{X}(t) + \mathcal{B}_0u(t),\phi\rangle_{D(\mathcal{A}_0^*)',D(\mathcal{A}_0^*)} 
= \langle \mathcal{X}(t),\mathcal{A}_0^*\phi\rangle_{\mathcal{H}^0} + u(t)\,\mathcal{B}_0^*\phi .
\end{equation*}
Besides, taking $\mathcal{X}(t,\cdot)=(y(t,\cdot),z(t,\cdot),\partial_tz(t,\cdot))^\top$ a sufficiently regular solution of \eqref{eq: cascade equation} (so that $\dot{\mathcal{X}}(t)\in\mathcal{H}^0$), and denoting $\phi=(\phi_3,\phi_3,\phi_3)^\top$, a straightforward computation using \eqref{eq: cascade equation} and integrations by parts shows that
$$
\langle\dot{\mathcal{X}}(t),\phi\rangle_{\mathcal{H}^0} = \langle \mathcal{X}(t),\mathcal{A}_0^*\phi\rangle_{\mathcal{H}^0} + u(t)\,\phi_3(L).
$$
Since $\langle\dot{\mathcal{X}}(t),\phi\rangle_{D(\mathcal{A}_0^*)',D(\mathcal{A}_0^*)}  = \langle\dot{\mathcal{X}}(t),\phi\rangle_{\mathcal{H}^0}$, a density argument finally leads to
$$
\mathcal{B}_0^*\phi = \phi_3(L)\qquad\forall\phi=(\phi_1,\phi_2,\phi_3)^\top\in D(\mathcal{A}_0^*).
$$
This formula defines the control operator $\mathcal{B}_0\in L(\mathbb{R},D(\mathcal{A}_0^*)')$ by transposition.
Following \cite[Section 5.1.4]{trelat_SB} or \cite[Proposition 10.9.1]{tucsnakweiss}, one could then define $\mathcal{B}_0$ with an abstract formula, but this is not useful in this paper.

\subsection{Well-posedness and admissibility property}\label{sec_app_B}
\noindent Following \cite{trelat_SB, tucsnakweiss}, we say that the control system \eqref{abstract_system} is well-posed in $\mathcal{H}^0$ if, for all $T>0$ and $u\in L^2(0,T)$, any solution of \eqref{abstract_system} such that $\mathcal{X}(0)\in\mathcal{H}^0$ satisfies $\mathcal{X}(t)\in\mathcal{H}^0$ for any $t>0$. Equivalently, we say that the control operator $\mathcal{B}_0$ is \emph{admissible}.
By \cite[Proposition 5.13 and Remark 5.56]{trelat_SB} or \cite[Theorem 4.4.3]{tucsnakweiss}, $\mathcal{B}_0$ is admissible if and only, for some $T>0$ (and equivalently, for any $T>0$), there exists $K_T>0$ such that 
$\int_0^T \vert \mathcal{B}_0^*\mathcal{X}(t) \vert^2\, \mathrm{d}t \leq K_T \Vert \mathcal{X}(0)\Vert_{\mathcal{H}^0}^2$, i.e., 
\begin{equation}\label{adm}
\int_0^T \vert \mathcal{X}^3(t,L)\vert^2\, \mathrm{d}t 
\leq K_T\left( \Vert \mathcal{X}^1(0)\Vert_{L^2}^2 + \Vert\partial_x \mathcal{X}^2(0)\Vert_{L^2}^2 + \Vert \mathcal{X}^3(0)\Vert_{L^2}^2 \right)
\end{equation}
for any solution $t\mapsto \mathcal{X}(t,\cdot)=(\mathcal{X}^1(t,\cdot),\mathcal{X}^2(t,\cdot),\mathcal{X}^3(t,\cdot))^\top$ of the dual system $\dot{\mathcal{X}}(t) = \mathcal{A}_0^*\mathcal{X}(t)$, i.e., of
\begin{subequations}\label{dual_system}
    \begin{align}
        &\partial_t \mathcal{X}^1 = \partial_{xx} \mathcal{X}^1 + c \mathcal{X}^1 , && \mathcal{X}^1(t,0) = \mathcal{X}^1(t,L) = 0, \label{dual_system_1} \\
        &\partial_t \mathcal{X}^2 = P_\beta \mathcal{X}^1 - \mathcal{X}^3 , && \mathcal{X}^2(t,0) = \partial_x \mathcal{X}^2(t,L) = 0, \label{dual_system_2} \\
        &\partial_t \mathcal{X}^3 = -\partial_{xx} \mathcal{X}^2 , && \mathcal{X}^3(t,0) = 0 . \label{dual_system_3} 
    \end{align}
\end{subequations}
Using \eqref{dual_system_2} and \eqref{dual_system_3}, we have $\partial_{tt}\mathcal{X}^3=\partial_{xx}\mathcal{X}^3-\beta \mathcal{X}^1$ with $\mathcal{X}^3(t,0) = \partial_x \mathcal{X}^3(t,L)=0$ (the latter, because $\partial_t\partial_x \mathcal{X}^2(t,L)=0$) and $\mathcal{X}^2(t,x) = - \int_0^x \int_\tau^L \partial_t \mathcal{X}^3(s,x) \,\mathrm{d}s\,\mathrm{d}\tau$.
Since $\Vert\partial_x \mathcal{X}^2(0)\Vert_{L^2} = \Vert\partial_t \mathcal{X}^3(0)\Vert_{H^{-1}_{(0)}}$, where $H^{-1}_{(0)}(0,L)$ is the dual of $H^1_{(0)}(0,L)$ with respect to the pivot space $L^2(0,L)$, the admissibility inequality \eqref{adm} is equivalent to
\begin{equation}\label{adm1}
\int_0^T \vert \mathcal{X}^3(t,L)\vert^2\, \mathrm{d}t 
\leq K_T\Big( \Vert \mathcal{X}^1(0)\Vert_{L^2}^2 + \Vert \mathcal{X}^3(0)\Vert_{L^2}^2 + \Vert\partial_t \mathcal{X}^3(0)\Vert_{H^{-1}_{(0)}}^2 \Big)
\end{equation}
for any solution of
\begin{subequations}\label{dual_system1}
\begin{align}
        &\partial_t \mathcal{X}^1 = \partial_{xx} \mathcal{X}^1 + c \mathcal{X}^1 ,  \label{dual_system1_1} \\
        &\partial_{tt} \mathcal{X}^3 = \partial_{xx} \mathcal{X}^3 - \beta \mathcal{X}^1 ,  \label{dual_system1_2}     \\
        &\mathcal{X}^1(t,0) = \mathcal{X}^1(t,L) = \mathcal{X}^3(t,0) = \partial_x \mathcal{X}^3(t,L) = 0 . \label{dual_system1_3} 
\end{align}
\end{subequations}

\begin{lemma}\label{lem_adm}
The admissibility inequality \eqref{adm1} holds.
Therefore, the control system \eqref{abstract_system} is well-posed in $\mathcal{H}^0$.
\end{lemma}

\smallskip

\begin{proof}
We first establish \eqref{adm1} with $\mathcal{X}^1=0$. We have, then, a 1D wave equation with Dirichlet condition at the left boundary and Neumann condition at the right boundary. The eigenvalues of the corresponding Laplacian operator are $\lambda_k=\frac{\pi}{2L}+\frac{k\pi}{L}$, for $k\in\mathbb{N}$, with associated eigenfunctions $\sqrt{\frac{2}{L}}\sin(\lambda_k t)$. Now, expanding $\mathcal{X}^3(t,x)$ as a spectral Fourier series in this basis of eigenfunctions, and using the Parseval theorem as in \cite[Section 5.2.4.1]{trelat_SB}, \eqref{adm1} follows, with $\mathcal{X}^1=0$.

Now, we treat $\mathcal{X}^1$ as a source term in the wave equation \eqref{dual_system1_2}, and \eqref{adm} follows by the Duhamel formula.
\end{proof}

\subsection{Ingham-M\"untz inequality}\label{sec_IM}
\noindent 
In this section, we report on the Ingham-M\"untz inequality, which is intrumental in this paper. 

\begin{theorem}\label{thm_IM}
Let $(\lambda_{1,n})_{n\in\mathbb{N}^*}$ be a sequence of distinct complex numbers. We assume that there exist $\alpha>1$ such that the sequence $(\lambda_{1,n}/n^\alpha)_{n\in\mathbb{N}^*}$ is bounded, and that there exist $n_0\in\mathbb{N}^*$, $C_1,C_2>0$ and such that:
$$
-\mathrm{Re}(\lambda_{1,n}) \geq C_1\, \mathrm{Im}(\lambda_{1,n})
\qquad \forall n\in\mathbb{N}^*\,\mid\, n\geq n_0
$$
and
$$
\vert\lambda_{1,n}-\lambda_{1,n'}\vert\geq C_2\vert n^\alpha-{n'}^\alpha\vert
\qquad \forall n,n'\in\mathbb{N}^*\,\mid\, n,n'\geq n_0 .
$$
Let  $(\lambda_{2,m})_{m\in\mathbb{Z}}$ be another sequence of distinct complex numbers. We assume that there exist $m_0\in\mathbb{N}$, $\gamma>0$, $z_0\in\mathbb{C}$ and $(\mu_m)_{m\in\mathbb{Z}}\in\ell^2(\mathbb{Z})$ such that 
$$
\lambda_{2,m} = \gamma\, i\,m + z_0 + \mu_k
\qquad \forall m\in\mathbb{Z} \,\mid\,\vert m\vert\geq m_0 .
$$
Then, for any $T>\frac{2\pi}{\gamma}$, there exists $C_T>0$ (depending only on $T$) such that
\begin{equation}\label{IM}
\int_0^T \Big\vert \sum_{n\in\mathbb{N}^*} a_n e^{\lambda_{1,n}t} + \sum_{m\in\mathbb{Z}} b_m e^{\lambda_{2,n}t} \Big\vert^2 \, \mathrm{d}t 
\geq  C_T \bigg( \sum_{n\in\mathbb{N}^*} \vert a_n\vert^2 e^{2\lambda_{1,n}T} + \sum_{m\in\mathbb{Z}} \vert b_m\vert^2 \bigg)
\end{equation}
for all sequences $(a_n)_{n\in\mathbb{N}^*}, (b_m)_{m\in\mathbb{Z}}\in\ell^2(\mathbb{Z})$.
\end{theorem}

\smallskip

The inequality \eqref{IM} combines the celebrated Ingham inequality \cite{Ingham} dealing with purely imaginary (hyperbolic) spectrum having a spectral gap and the M{\"u}ntz-Sz\'asz theorem \cite{avdoninivanov} that tackles the parabolic spectrum.
It was first discovered and established in \cite{{zhang2003polynomial}}, \cite[Lemma 4.1]{zhang2004polynomial}, which are the papers that have inspired the present article, as mentioned at the beginning. The Ingham-M{\"u}ntz inequality has then been generalized in \cite[Theorem 1.1]{komornik2015ingham} and more recently in \cite[Proposition 1.7]{bhandari2024boundary} (in which it is referred to as a combined Ingham-type inequality). 

Note that, by the assumptions done on the sequence $(\lambda_{1,n})_{n\in\mathbb{N}^*}$, the term at the right-hand side of \eqref{IM} can be replaced by $C_T \big( \sum_{n\in\mathbb{N}^*} \vert a_n\vert^2 e^{-2n^\alpha T} + \sum_{m\in\mathbb{Z}} \vert b_m\vert^2 \big)$ (maybe, with a different constant $C_T$).
This inequality is ``almost optimal" with respect to the power $\alpha$ in the weight $e^{-2n^\alpha T}$ of $\vert a_n\vert^2$, and is ``optimal" with respect to the weight $1$ of $\vert b_m\vert^2$, in the following sense.


First, taking all $a_n=0$ in \eqref{IM}, we recover the Ingham inequality \cite{Ingham}, which is well known to be optimal in the above sense; actually the converse inequality holds true (think of the Parseval theorem for instance). Hence, in \eqref{IM}, the weight $1$ of $\vert b_m\vert^2$ is optimal. For example, the inequality fails with the weight $1/m$.

Second, taking all $b_m=0$ in \eqref{IM}, we recover an inequality resulting from the M{\"u}ntz-Sz\'asz theorem \cite{avdoninivanov}, which can be improved: one can replace $\alpha$ by $\alpha/2$, i.e., take the weight $e^{-\mathrm{Cst}\,n^{\alpha/2}\sqrt{1+T}}$ (or equivalently, $e^{-\mathrm{Cst}\sqrt{-\lambda_{1,n}}\sqrt{1+T}}$) for $\vert a_n\vert^2$. This fact was used in the past to establish exact null controllability properties for the 1D heat equation (see \cite[Section 7]{russell1978controllability}). More generally, the observability inequality for the multi-D Dirichlet heat equation on a domain $\Omega$ with internal observation on $\omega\subset\Omega$, which holds for any $T>0$, is usually written as $\int_0^T \Vert\mathds{1}_\omega e^{t\triangle}f\Vert_{L^2(\Omega)}^2\, dt \geq C\Vert e^{T\triangle}f\Vert_{L^2(\Omega)}^2$ for any $f\in L^2(\Omega)$, for some $C>0$ not depending on $f$. But actually, as noticed in \cite[Section 6]{fernandezcarazuazua}, this inequality can be improved by replacing, at the right-hand side, $\Vert e^{T\triangle}f\Vert$ by $\Vert e^{-C\sqrt{1+T}\sqrt{-\triangle}}f\Vert$, and then one obtains an ``optimal" inequality in the above sense, i.e., the power $1/2$ of $-\triangle$ cannot be decreased: this is known by the so-called Lebeau-Robbiano spectral inequalities (see \cite[Theorem 3]{lebeauzuazua}).

According to the above remarks, it is natural to wonder whether the complete Ingham-M\"untz inequality \eqref{IM} remains valid with the weight $e^{-2n^\alpha T}$ replaced by $e^{-\mathrm{Cst}\, n^{\alpha/2}\sqrt{1+T}}$. This fact is not known. But at least, according to the above discussion, the inequality \eqref{IM} fails if one replaces the weight $e^{-2n^\alpha T}$ replaced by $e^{-\mathrm{Cst}\, n^{c\alpha}}$ with $c<1/2$.
This is in this sense that we have written that the Ingham-M\"untz inequality \eqref{IM} is ``almost optimal".

\paragraph{Acknowlegment.}
The third author acknowledges the support of ANR-20-CE40-0009 (TRECOS).


\bibliographystyle{IEEEtranS}
\nocite{*}
\bibliography{IEEEabrv,mybibfile}

\end{document}